\documentclass[a4paper,12pt]{article}
\usepackage{amssymb}
\usepackage{bbding}
\usepackage[top=2.5cm,bottom=2.5cm,left=2.5cm,right=2.5cm]{geometry}
\usepackage{amsmath,amssymb,amsthm,graphics,latexsym,amsfonts}
\usepackage{caption}
\usepackage{booktabs}
\usepackage{graphicx}
\usepackage{bm,bbm}
\usepackage{mathrsfs}
\usepackage{tabularx}

\newtheorem{thm}{Theorem}[section]
\newtheorem{lem}[thm]{Lemma}
\newtheorem{cor}[thm]{Corollary}

\newtheorem{cla}{Claim}
\newtheorem{clar}{Claim}
\newtheorem{cla2}{Claim}

\makeatletter
\newcommand{\Rmnum}[1]{\expandafter\@slowromancap\romannumeral #1@} 
\makeatother
\allowdisplaybreaks[3] 

\begin{document}

\title{{On the forcing spectrum of generalized Petersen graphs $P(n,2)$}\thanks{Supported by NSFC (grant no. 11371180).}}
\author{Shuang Zhao, Jinjiang Zhu, Heping Zhang\thanks{Corresponding author.\protect\\ \hspace*{0.63cm}E-mail addresses: zhaosh08@lzu.edu.cn (S. Zhao), 759483117@qq.com (J. Zhu), zhanghp@lzu.edu.cn (H. Zhang).}\\
{\small School of Mathematics and Statistics, Lanzhou University, Lanzhou, Gansu 730000, China}}
\date{}
\maketitle

\begin{abstract}
    The forcing number of a perfect matching $M$ of a graph $G$ is the smallest cardinality of subsets of $M$ that are contained in no other perfect matchings of $G$. The forcing spectrum of $G$ is the collection of forcing numbers of all perfect matchings of $G$. In this paper, we classify the  perfect matchings of a generalized Petersen graph  $P(n,2)$ in two types, and show that the forcing spectrum is the union of two integer intervals. For $n\ge 34$, it is $\left[\lceil \frac { n }{ 12 }  \rceil+1,\lceil \frac { n+3 }{ 7 }  \rceil +\delta (n)\right]\cup \left[\lceil \frac { n+2 }{ 6 }  \rceil,\lceil \frac { n }{ 4 }  \rceil\right]$, where $\delta (n)=1$ if $n\equiv 3$ (mod 7), and $\delta (n)=0$ otherwise.
\end{abstract}

\section{Introduction}
    The forcing number of a perfect matching of hexagonal systems has been introduced by Harary et al. \cite{original} in 1991. The roots of this concept can be found in earlier literatures by Randi\'{c} and Klein \cite{early,RK} using the name `innate degree of freedom', which plays an important role in the resonance theory of theoretic chemistry. For more details, we refer the reader to \cite{chechen}.

    Let $G$ be a graph with vertex set $V(G)$ and edge set $E(G)$. A \emph{perfect matching} $M$ of $G$ is a set of disjoint edges that covers all vertices of $G$. A \emph{forcing set} $S$ of $M$ is a subset of $M$ such that $S$ is contained in no other perfect matchings of $G$. Namely, the subgraph $G-V(S)$, which is obtained from $G$ by deleting ends of all edges in $S$, is empty (with no vertices) or has a unique perfect matching. The \emph{forcing number of $M$}, denoted by $f(G,M)$, is the smallest cardinality over all forcing sets of $M$. A cycle of $G$ is called \emph{$M$-alternating} if its edges appear alternately in $M$ and $E(G)\setminus M$. There is an equivalent definition for a forcing set of a perfect matching as follows.

    \begin{thm}[\cite{1,15}]
        \label{lem5}
        Let $G$ be a graph with a perfect matching $M.$ Then a subset $S\subseteq M$ is a forcing set of $M$ if and only if each $M$-alternating cycle of $G$ contains at least one edge of $S.$
    \end{thm}

    From the theorem we can see that the forcing number $f(G,M)$ is bounded below by the maximum number of disjoint $M$-alternating cycles. Using the minimax theorem on feedback set of Lucchesi and Younger \cite{lucchesi} and Barahona et al. \cite{barahona}, Pachter and Kim \cite{squaregrids} pointed out the following conclusion.

    \begin{thm}[\cite{squaregrids}]
        \label{thm2}
        Let $G$ be a bipartite graph without $K_{3,3}$ minor$.$ Then for each perfect matching $M$ of $G,$ $f(G,M)=C(G,M),$ where $C(G,M)$ denotes the maximum number of disjoint $M$-alternating cycles in $G.$
    \end{thm}

    An extension of this theorem was given by Guenin and Thomas \cite{robin} using the minimax theorem on transversal.

    \begin{thm}[\cite{robin}]
        \label{lemmima}
        Let $G$ be a bipartite graph which contains no even subdivision of $K_{3,3}$ or the Heawood graph as a nice subgraph$.$ Then for each perfect matching $M$ of $G,$ $f(G,M)=C(G,M).$
    \end{thm}

    The \emph{maximum} (resp. \emph{minimum}) \emph{forcing number} of a graph $G$ is the maximum (resp. minimum) value of $f(G,M)$ over all perfect matchings $M$ of $G$. Adams et al. \cite{1} introduced the \emph{forcing spectrum} of $G$ as the collection of forcing numbers of all perfect matchings in $G$. To consider the distribution of forcing numbers of perfect matchings in $G$, the authors in \cite{zhao,zhao1} proposed the \emph{forcing polynomial} of $G$ as
    \begin{equation*}
        \label{equ1}
        F(G,x)=\sum_{M\in\mathcal{M}(G)}{{x}^{f(G,M)}},
    \end{equation*}
    where $\mathcal{M}(G)$ denotes the set of all perfect matchings of $G$.

    For a hexagonal system with a perfect matching, Xu et al. \cite{Xu} showed that the maximum forcing number is equal to the Clar number (i.e. the size of a maximum resonant set), which can measure the stability of benzenoid hydrocarbons, and Zhou and Zhang \cite{zhou} proved that for each perfect matching $M$ with the maximum forcing number, there exists a maximum resonant set consisting of disjoint $M$-alternating hexagons. Some stronger results hold for polyomino graphs \cite{zhoux1,zhoux2}. For a hexagonal system with minimum forcing number one \cite{5,20,19}, the forcing spectrum form either the integer interval from one to the Clar number or with only the gap two \cite{deng}. By Theorem \ref{thm2}, Pachter and Kim \cite{squaregrids} and Afshani et al. \cite{2} gave the forcing spectrum of square grids $P_{2n}\times P_{2n}$ as an integer interval $\left[n,n^2\right]$. By introducing the trailing vertex method, Riddle \cite{15} presented the minimum forcing numbers of tori $C_{2m}\times C_{2n}$ and hypercubes $Q_k$ with even $k$, and Wang et al. \cite{yedong} derived the minimum forcing number of toroidal polyhexes. Sharp lower bounds for minimum forcing numbers of  boron-nitrogen fullerene graphs and fullerene graphs were obtained in \cite{jiangxiao} and \cite{22}, respectively. Furthermore, the maximum forcing numbers of some graphs have been studied, such as stop signs \cite{stopsigns}, rectangle grids $P_{m}\times P_{n}$ \cite{2}, cylindrical girds $P_{m}\times C_{n}$ \cite{2,jiang2}, and tori $C_{2m}\times C_{2n}$ \cite{9}. Recently, Lei et al. \cite{Lei} put forward the anti-forcing number of a perfect matching of a graph, and showed that for a perfect matching of a graph the anti-forcing number is no less than the forcing number. For the anti-forcing spectrum of a graph, see \cite{deng1,deng2}.

    A \emph{generalized Petersen graph} $P(n,k)$ ($n\ge 3$, $1\le k\le n-1$) \cite{def} is a graph on $2n$ vertices with vertex set
    \[V(P(n,k))=\{ { u }_{ i },{ v }_{ i }:0\le i\le n-1 \},\]
    and edge set
    \[E(P(n,k))=\{ { u }_{ i }{ u }_{ i+k },{ u }_{ i }{ v }_{ i },{ v }_{ i }{ v }_{ i+1 }:0\le i\le n-1\}. \]
    Unless stated, the subscripts modulo $n$ in the following. The edges $u_iv_i$ are referred to \emph{spokes}. Some properties of $P(n,k)$ were studied, such as Hamilton connectivity \cite{per3}, domination number \cite{per1}, total coloring \cite{per2} and reliability \cite{newww}. Moreover, Schrag et al. \cite{lingyige} and Yu \cite{yu} showed that for $k\ge 3$, $P(n,k)$ is 2-extendable if and only if $n \neq 2k$ or $3k$; $P(n,2)$ is 2-extendable if and only if $n \neq 4,5,6,8$; $P(n,1)$ is 2-extendable if and only if $n$ is even.

    In this paper, we focus on generalized Petersen graph in the case of $k=2$. We always use $P(n)$ to stand for $P(n,2)$ in the following. In particular, $P(5)$ is the usual Petersen graph (see Fig. \ref{example}(a)). For convenience, we place $P(n)$ in a strip with the left side and right side identified as illustrated in Fig. \ref{example}(b). In the next section, we classify the  perfect matchings of $P(n)$ in two types, calculate the perfect matching count, and list the forcing polynomials of $P(n)$ for $3\le n \le 36$. By analysing properties of perfect matchings, we obtain two sets of forcing numbers of first and second types of perfect matchings in Sections 3 and 4, respectively, which are integer intervals. In particular, for $n\ge 11$, the forcing numbers of first type of perfect matchings are continuous from $\lceil \frac { n+2 }{ 6 } \rceil$ to $\lceil \frac { n }{ 4 } \rceil$; for $n\ge 34$, the forcing numbers of second type of perfect matchings are continuous from $\lceil \frac { n }{ 12 } \rceil+1$ to $\lceil \frac { n+3 }{ 7 } \rceil +\delta (n)$, where $\delta (n)=1$ if $n\equiv 3~(\text{mod } 7)$, and $\delta (n)=0$ otherwise. From the above conclusions, it follows that the forcing spectrum of $P(n)$ is continuous for $n=3,4,\ldots,58,59,66,73,80,87,94$, and has one gap for others $n$.

    \begin{figure}[htbp]
        \centering
        \includegraphics[height=0.71in]{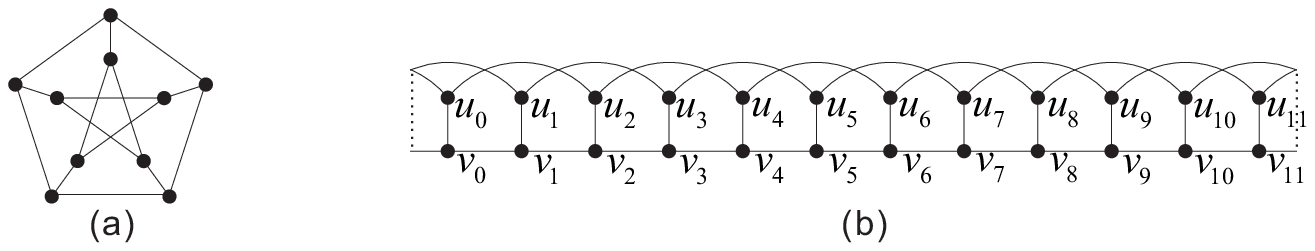}
        \caption{(a) The Petersen graph $P(5)$ and (b) $P(12)$.}
        \label{example}
    \end{figure}

\section{Some preliminaries}
    First we present some properties of a perfect matching of $P(n)$.

    Let $\mathcal{M}(P(n))$ be the set of perfect matchings of $P(n)$. For $M\in  \mathcal{M}(P(n))$, if there are no spokes in $M$, then $M$ should be one of the two perfect matchings illustrated with bold lines in Figs. \ref{case}(a) and (b). Unless stated, we use bold lines to denote the edges in a perfect matching in the following. If there is a spoke in $M$, then the number of spokes between any two consecutive spokes $u_iv_i$ and $u_jv_j$ ($i<j$) of $M$ is even (here the first spoke $u_lv_l$ can be considered as $u_{l+n}v_{l+n}$). This is because $v_kv_{k+1}\in M$ for $k=i+1,i+3,\ldots,j-2$. Note that if there is precisely one spoke $u_iv_i$ in $M$, then itself can be considered as two consecutive spokes $u_iv_i$ and $u_{i+n}v_{i+n}$, which implies that $n$ is odd.

    Moreover, let $u_iv_i$ and $u_jv_j$ ($i<j$) be two consecutive spokes in $M$. If $j-i-1\equiv 0$ (mod 4) (resp. $j-i-1\equiv 2$ (mod 4)), then the edges in $M$ incident with the vertices $u_k$ and $v_k$ must be the ones illustrated in Fig. \ref{case}(c) (resp. Fig. \ref{case}(d)) for $k=i+1,i+2,\ldots,j-1$. So either the number of spokes between any two consecutive spokes in $M$ is 0 (mod 4), or the number of spokes between any two consecutive spokes in $M$ is 2 (mod 4).

    \begin{figure}[htbp]
        \centering
        \includegraphics[height=0.89in]{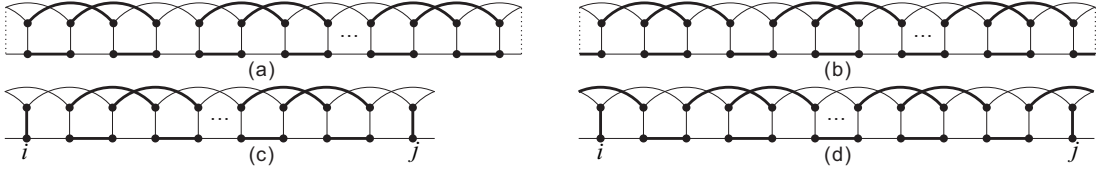}
        \caption{Two types of perfect matchings of $P(n)$.}
        \label{case}
    \end{figure}

    We now divide $\mathcal{M}(P(n))$ in two subsets: $\mathcal{M}_1(P(n))=\{M\in \mathcal{M}(P(n)):$ $M$ has a spoke and the number of spokes between any two consecutive spokes in $M$ is 0 (mod 4)\}$\cup$\{$M\in \mathcal{M}(P(n)):$ $M$ is illustrated in Fig. \ref{case}(a)\}, and $\mathcal{M}_2(P(n))=\{M\in \mathcal{M}(P(n)):$ $M$ has a spoke and the number of spokes between any two consecutive spokes in $M$ is 2 (mod 4)\}$\cup$\{$M\in \mathcal{M}(P(n)):$ $M$ is illustrated in Fig. \ref{case}(b)\}.

    We now count perfect matchings of $P(n)$ in each type.

    There are two ways $A$ and $B$ in Fig. \ref{stru1}(a) to classify the edges in a perfect matching in $\mathcal{M}_1(P(n))$. In detail, four edges $u_iu_{i+2},u_{i+1}u_{i+3},v_{i}v_{i+1},v_{i+2}v_{i+3}$ constitute a structure $A$, and one spoke $u_jv_j$ constitutes a structure $B$. Then each perfect matching in $\mathcal{M}_1(P(n))$ can be expressed by a (not necessarily unique) cyclic sequence of $A$ and $B$ with $4a+b=n$, where $a$ and $b$ denote the number of $A$ and $B$, respectively. Also, we use the notation $W^m$ to denote sequence $\underbrace{WW\cdots W}_m$, where $W$ is a sequence of $A$ and $B$. For example, the perfect matching of $P(26)$ in Fig. \ref{stru1}(b) can be expressed by $AABBBBABABBBBB$ or $A^2B^4(AB)^2B^4$.

    We define a \emph{chain} to be a vertex induced subgraph of $P(n)$ admitting a perfect matching expressed by a sequence of $A$ and $B$. Also, we could use the sequence to express the chain. As an example, a chain $P(n)[\{u_j,v_{j}:j=i,i+1,\ldots,i+10\}]$ with the perfect matching $AABBB$ (or briefly, a chain $AABBB$), is illustrated in Fig. \ref{stru1}(c). For $P(n)$ with a perfect matching expressed by a sequence of $A$ and $B$, we define a \emph{segment} $A$ (resp. $B$) to be an (inclusion-wise) maximal chain with perfect matching expressed by a sequence of $A$ (resp. $B$), and an \emph{$AB$-chain} to be a chain formed by a segment $A$ and its immediate right-hand segment $B$.

    \begin{figure}[htbp]
        \centering
        \includegraphics[height=1.67in]{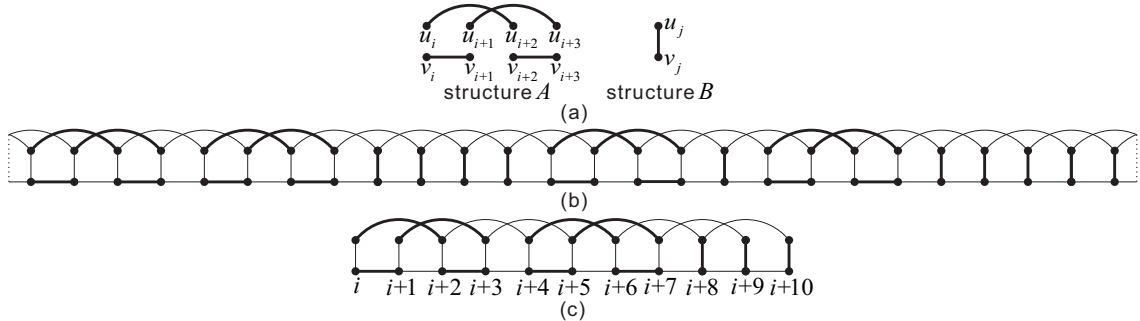}
        \caption{(a) Structures $A$ and $B$, (b) perfect matching $A^2B^4(AB)^2B^4$, (c) chain $AABBB$.}
        \label{stru1}
    \end{figure}

    Next we calculate $|\mathcal{M}_1(P(n))|$. By establishing a one-to-one correspondence between the set $\mathcal{M}_l$ $(l\ge 0)$ of perfect matchings in $\mathcal{M}_1(P(n))$ ($n\ge 5$) with $l$ spokes and the set $\mathcal{S}_l$ of ways to select $l$ balls in $n$ distinct balls arranged in a cycle such that the number of balls between any two consecutive selected balls is 0 (mod 4), we have $|\mathcal{M}_l|=|\mathcal{S}_l|$.

    Suppose there are $l+\frac { n-l }{ 4 }$ distinct boxes. Select $l$ boxes from them, and denote each by $\mathcal{B}$ and each of the others by $\mathcal{A}$. Then the total number of such selections is $\binom{l+\frac { n-l }{ 4 }}{l}$. Obviously, each selection corresponds to a unique sequence of $\mathcal{A}$ and $\mathcal{B}$, denoted by $Q_i$ ($1\le i\le \binom{l+\frac { n-l }{ 4 }}{l}$) respectively, which can express a (not necessarily unique) perfect matching in $\mathcal{M}_l$.

    Pick a $Q_i$, and put four balls in each box $\mathcal{A}$ and one ball in each box $\mathcal{B}$. Then the total number of balls is $n$. Place the $n$ balls in a line with the same order as boxes. Copy $Q_i$ $n$ times to get $Q_i^1,Q_i^2,\ldots,Q_i^n$. Label the balls in each $Q_i^j$ as $j,j+1,\ldots,n,1,2,\ldots,j-1$ for $j=1,2,\ldots,n$. Deal with others $Q_k$ by the same way as above to get $n\cdot \binom{l+\frac { n-l }{ 4 }}{l}$ labels. Obviously, each label corresponds to a unique way in $\mathcal{S}_l$, and naturally, a unique perfect matching in $\mathcal{M}_l$ as well.

    It is easy to see that each perfect matching in $\mathcal{M}_l$ (expressed by $i_1i_2\ldots i_{l+\frac { n-l }{ 4 }}$ with $i_j\in \{A,B\}$, $j=1,2,\ldots,l+\frac {n-l} 4$) coincides with some labels $Q_{k_1}^{j_1},Q_{k_2}^{j_2},\ldots,Q_{k_w}^{j_w}$. To count the total number $w$, we define the \emph{period} $p$ of $i_1i_2\ldots i_{l+\frac { n-l }{ 4 }}$ to be $\min\{p\ge 1:i_j=i_{j+p~(\text{mod}~ l+\frac { n-l }{ 4 })}\text{ for } j=1,2,\ldots,l+\frac { n-l }{ 4 }\}$. Then there are $p$ sequences from $\{Q_i:i=1,2,\ldots,\binom{l+\frac { n-l }{ 4 }}{l}\}$ being able to express $M$. Furthermore, for each of the above sequences $Q_k$, there are $m:=\frac{l+\frac { n-l }{ 4 }}{p}$ labels $Q_{k}^{j_1},Q_{k}^{j_2},\ldots,Q_{k}^{j_m}$ coinciding with $M$. Then the times of repetitions of $M$ in $n\cdot \binom{l+\frac { n-l }{ 4 }}{l}$ labels is $p\cdot \frac{l+\frac { n-l }{ 4 }}{p}=l+\frac { n-l }{ 4 }$, which implies $|\mathcal{M}_l|=\frac { n }{ l+\frac { n-l }{ 4 }  } \binom{l+\frac { n-l }{ 4 }}{l}$. Using the initial cases of $n=3$ and 4 from Table \ref{tab}, we have the following formula.

    \begin{thm}
        \label{thm1num}
        For $n\ge 3,$ $|\mathcal{M}_1(P(n))|=\left\{ \begin{array} {ll} 2 & \text{if}~n=4, \\  \sum \limits_{ l=0,\ n-l\equiv 0~(\text{\emph{mod} }4)  }^{ n }{\frac { n }{ l+\frac { n-l }{ 4 }  } \binom{l+\frac { n-l }{ 4 }}{l}} & \text{otherwise.}\end{array} \right.$
    \end{thm}

    Similarly, there are also two ways $C$ and $D$ in Fig. \ref{stru2}(a) to classify the edges in a perfect matching in $\mathcal{M}_2(P(n))$. In detail, three edges $u_iu_{i+2},u_{i+1}v_{i+1},v_{i+2}v_{i+3}$ constitute a structure $C$, and four edges $u_ju_{j+2},u_{j+1}u_{j+3},v_{j+1}v_{j+2},v_{j+3}v_{j+4}$ constitute a structure $D$. Then each perfect matching in $\mathcal{M}_2(P(n))$ can be expressed by a (not necessarily unique) cyclic sequence of $C$ and $D$ with $3c+4d=n$, where $c$ and $d$ denote the number of $C$ and $D$, respectively. For example, the perfect matching of $P(25)$ in Fig. \ref{stru2}(b) can be expressed by $CDDDDCC$ or $CD^4C^2$. A vertex induced subgraph of $P(n)$ admitting a perfect matching expressed by a sequence of $C$ and $D$ is also referred to \emph{chain}.

    \begin{figure}[htbp]
        \centering
        \includegraphics[height=1.1in]{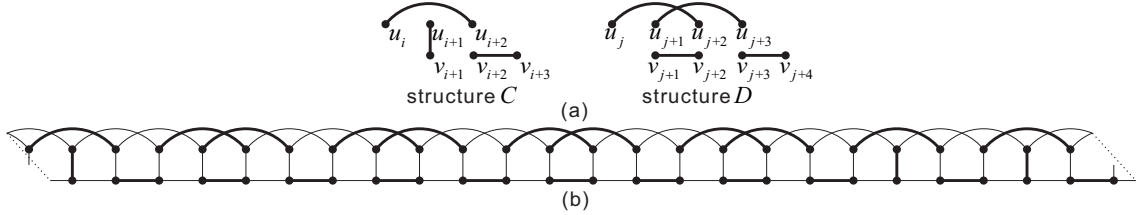}
        \caption{(a) Structures $C$ and $D$ and (b) perfect matching $CD^4C^2$.}
        \label{stru2}
    \end{figure}

    Similar to Theorem \ref{thm1num}, we have the following formula on $|\mathcal{M}_2(P(n))|$.

    \begin{thm}
        \label{thm2num}
        For $n\ge 3,$ $|\mathcal{M}_2(P(n))|=\left\{ \begin{array} {ll} 1 & \text{if}~n=4, \\  \sum \limits_{  l=0, \ n-3l\equiv 0~(\text{\emph{mod} }4) }^{ \lfloor \frac { n }{ 3 } \rfloor }{\frac { n }{ l+\frac { n-3l }{ 4 }  } \binom{l+\frac { n-3l }{ 4 }}{l}} & \text{otherwise.}\end{array} \right.$
    \end{thm}

    By a computer program, we obtain explicit expressions for forcing polynomial of $P(n)$ for $n=3,4,\ldots,36$ listed in Table \ref{tab}, where the first sum form is $\sum \limits_{ M\in { \mathcal{M} }_1(P(n))}{ { x }^{ f(P(n),M) } }$ and the second sum form is $\sum \limits _{ M\in { \mathcal{M} }_2(P(n))}{ { x }^{ f(P(n),M) } }$.

    \begin{table}[htbp]\scriptsize
        \centering
        \caption{\label{tab}{Forcing polynomial of $P(n)$ $(3\le n \le 36)$}.}
        \begin{tabular}{clcl}
            \toprule
            $n$ & forcing polynomial& $n$ & forcing polynomial \\
            \midrule
            3 & $(x^2)+(3x)$ & 20 & $(330x^5+300x^4)+(34x^4+20x^3)$ \\
            4 & $(2x)+(x)$  & 21 & $(x^6+742x^5+126x^4)+(70x^4+3x^3)$\\
            5 & $(6x^2)+(0)$& 22 & $(133x^6+1034x^5+33x^4)+(66x^4+11x^3)$\\
            6 & $(7x^2)+(3x^2)$ & 23 & $(300x^6+1357x^5)+(69x^4+23x^3)$\\
            7 & $(8x^2)+(7x^2)$ & 24 & $(859x^6+1428x^5)+(24x^5+100x^4+3x^3)$\\
            8 & $(4x^3+9x^2)+(4x^3)$ & 25 & $(x^7+2150x^6+1005x^5)+(150x^4)$\\
            9 & $(x^3+18x^2)+(3x^2)$& 26 & $(287x^7+3523x^6+546x^5)+(169x^4)$\\
            10 & $(26x^3)+(10x^3)$ & 27 & $(757x^7+5013x^6+243x^5)+(54x^5+165x^4)$\\
            11 & $(34x^3)+(11x^2)$ & 28 & $(2203x^7+6041x^6+56x^5)+(193x^5+84x^4)$\\
            12 & $(47x^3)+(4x^3+3x^2)$ & 29 & $(x^8+6119x^7+5336x^6)+(203x^5+116x^4)$\\
            13 & $(x^4+65x^3)+(13x^3)$& 30 & $(617x^8+11335x^7+3860x^6)+(205x^5+183x^4)$\\
            14 & $(57x^4+35x^3)+(21x^3)$ & 31 & $(1861x^8+17422x^7+2542x^6)+(31x^6+372x^5+93x^4)$\\
            15 & $(91x^4+35x^3)+(18x^3)$& 32 & $(5789x^8+23008x^7+1328x^6)+(564x^5+32x^4)$\\
            16 & $(125x^4+48x^3)+(20x^3)$ & 33 & $(x^9+17237x^8+23936x^7+407x^6)+(693x^5+14x^4)$\\
            17 & $(x^5+238x^4)+(17x^4+17x^3)$ & 34 & $(1327x^9+35547x^8+20468x^7+51x^6)+(85x^6+765x^5+34x^4)$\\
            18 & $(61x^5+270x^4)+(39x^3)$ & 35 & $(4516x^9+58842x^8+15860x^7)+(427x^6+630x^5+35x^4)$\\
            19 & $(153x^5+304x^4)+(38x^3)$ & 36 & $(15137x^9+83790x^8+10416x^7)+(508x^6+792x^5+3x^4)$\\
            \bottomrule
        \end{tabular}
    \end{table}

    We now describe a method to test whether a graph has a unique perfect matching.

    It is well known that a bipartite graph with a unique perfect matching contains a pendant edge (with an end of degree one) (see \cite{lovasz}). Kotzig \cite{kotzig} showed that if a connected graph has a unique perfect matching, then the graph has a cut edge in the perfect matching. Some immediate consequences of the above results are as follows.

    \begin{thm}[\cite{unique}]
        A connected graph $G$ has a unique perfect matching if and only if\\
        $(1)$ $G$ has a cut edge $e$ such that $G-e$ has an odd component$,$ and\\
        $(2)$ when the ends of the cut edge are deleted$,$ the resulting subgraph (if nonempty) has a unique perfect matching$.$
    \end{thm}

    \begin{cor}
        \label{structure}
        Given a perfect matching $M$ of a graph $G$ and $S\subseteq M.$ If we can reach an empty graph while deleting recursively the ends of one of the following edges from $G-V(S),$ then $S$ is a forcing set of $M,$ and vise versa$:$\\
        $(1)$ pendant edge$;$\\
        $(2)$ cut edge $e$ in some component $H$ such that $H-e$ has an odd component$;$\\
        $(3)$ edge which can be determined to belong to all perfect matchings of $G-V(S).$
    \end{cor}

    In the following, we always use Corollary \ref{structure} to test whether an edge subset $S$ is a forcing set of a perfect matching $M$ of a graph $G$. For convenience, we denote the edge whose ends are deleted in $i$-th step by $e_i(G,M,S)$, and $E_i(G,M,S)=\{e_j(G,M,S):j=1,2,\ldots,i\}$ (or briefly, $e_i$ and $E_i$ if there is no ambiguity). For some applications of Corollary \ref{structure}, see Claim \ref{claim11} in Theorem 3.1, Claim \ref{claim33} in Theorem 3.2, Claims \ref{minkk} and \ref{last} in Theorem 4.1.

\section{Forcing number of a perfect matching in $\mathcal{M}_1(P(n))$}
    In this section, we first derive the maximum and minimum forcing numbers of first type of perfect matchings, then prove the continuity. In detail, for $n\ge 11$, the set of forcing numbers of first type of perfect matchings form the integer interval $\left[\lceil \frac { n+2 }{ 6 }  \rceil,\lceil \frac { n }{ 4 }  \rceil\right]$.

\subsection{Maximum value of forcing numbers}
    \begin{thm}
        \label{thm1max}
        The maximum value of forcing numbers of perfect matchings in $\mathcal{M}_1(P(n))$ is $\lceil \frac { n }{ 4 } \rceil$ for $n\ge 9.$
    \end{thm}
    \begin{proof}
        We divide our proof in two steps. First we find a perfect matching in $\mathcal{M}_1(P(n))$ with forcing number $\lceil \frac { n }{ 4 } \rceil$. Then we prove that the forcing number of each perfect matching in $\mathcal{M}_1(P(n))$ is no more than $\lceil \frac { n }{ 4 } \rceil$.

        (1) Let $M=B^n$ and $S=\{u_{4i}v_{4i}:i=0,1,\ldots, \lceil \frac { n }{ 4 } \rceil -1\}\subseteq M$ (see Fig. \ref{claim1}(a)). Unless stated, we use double lines to denote the edges in a forcing set in the following.

        \begin{clar}
            \label{claim11}
            $S$ is a forcing set of $M$ of $P(n).$
        \end{clar}

        We prove it by Corollary \ref{structure} (1). Since $u_{4i+2}v_{4i+2}$ is a pendant edge of $P(n)-\{u_{4i},u_{4i+4}\}$, we can determine $E_{\lceil \frac { n }{ 4 } \rceil -1}=\{u_{4i+2}v_{4i+2}:i=0,1,\ldots, \lceil \frac { n }{ 4 } \rceil -2\}$. Since $u_{4i+1}v_{4i+1}$ and $u_{4i+3}v_{4i+3}$ are pendant edges of $P(n)-\{v_{4i},v_{4i+2},v_{4i+4}\}$, we can determine $E_{3\lceil \frac { n }{ 4 } \rceil -3}=E_{\lceil \frac { n }{ 4 } \rceil -1}\cup\{u_{4i+1}v_{4i+1},u_{4i+3}v_{4i+3}:i=0,1,\ldots, \lceil \frac { n }{ 4 } \rceil -2\}$. For the resulting graph $G-V(S)-V(E_{3\lceil \frac { n }{ 4 } \rceil -3})$, by a similar argument as above, we could reach an empty graph. Then $S$ is a forcing set of $M$.

        Suppose there is a forcing set $S_0$ of $M$ with cardinality less than $\lceil \frac { n }{ 4 } \rceil$. Then there are four continuous spokes $u_iv_i,$ $u_{i+1}v_{i+1},$ $u_{i+2}v_{i+2},$ $u_{i+3}v_{i+3}$ in $M$ but not in $S_0$. Hence an $M$-alternating cycle $u_iu_{i+2}v_{i+2}v_{i+3}u_{i+3}u_{i+1}v_{i+1}v_iu_i$ illustrated with dotted cycle in Fig. \ref{claim1}(b) contains no edges of $S_0$, a contradiction to Theorem \ref{lem5}.

        \begin{figure}[ht]
            \centering
            \includegraphics[height=1.35in]{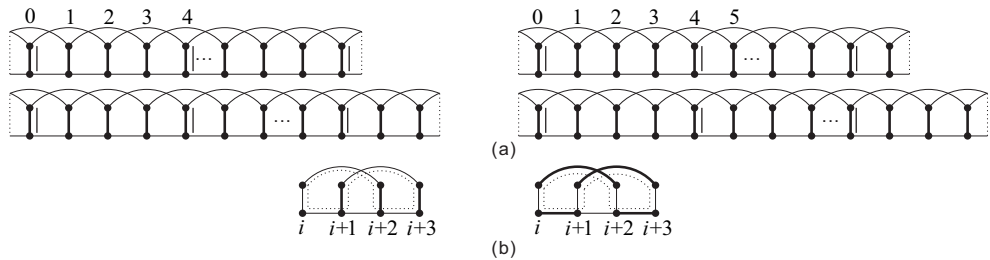}
            \caption{Perfect matching which achieves the upper bound of Theorem \ref{thm1max}.}
            \label{claim1}
        \end{figure}

        (2) Next we prove that for each $M\in \mathcal{M}_1(P(n))$, there is a forcing set of $M$ with cardinality no more than $\lceil \frac { n }{ 4 }\rceil$. From the above discussions, we assume that $M$ can be expressed by a sequence of at least one $A$. We now consider the  following cases of $M$.

        \textbf{Case 1.} There are no segments $B$ with 2 or 3 (mod 4) $B$ in $P(n)$.

        \textbf{Case 1.1.} There are no segment $B$ in $P(n)$. Then $n\equiv 0$ (mod 4) and $M=A^{\frac n 4}$. Similar to the proof above, we can confirm that the edge subset in Fig. \ref{algo1}(a) with cardinality $\frac { n }{ 4 }$ is a minimum forcing set of $M$.

        \textbf{Case 1.2.} There is a segment $B$ with at least four $B$ in $P(n)$. Then for each $AB$-chain $W_j$, we give the edge subset $S_j$ in Fig. \ref{algo1}(b). Similar to the proof of Claim \ref{claim11}, we can confirm that $\cup_{j} S_j$ with cardinality no more than $\lceil \frac { n }{ 4 } \rceil$ is a forcing set of $M$.

        \textbf{Case 1.3.} Each segment $B$ has precisely one $B$ in $P(n)$. Then there exits a chain $ABA$ (see Fig. \ref{algo1}(c)), denoted by $W$. First for $W$, we give the edge subset $S_0$ in Fig. \ref{algo1}(c). Then in turn for other $j$-th chain $A$, we give the edge subset $S_j$ in Fig. \ref{algo1}(d). Similar to the proof of Claim \ref{claim11}, we can confirm that $\cup_jS_j\cup S_0$ with cardinality no more than $\lceil \frac { n }{ 4 } \rceil$ is a forcing set of $M$.

        \begin{figure}[htbp]
            \centering
            \includegraphics[height=1.71in]{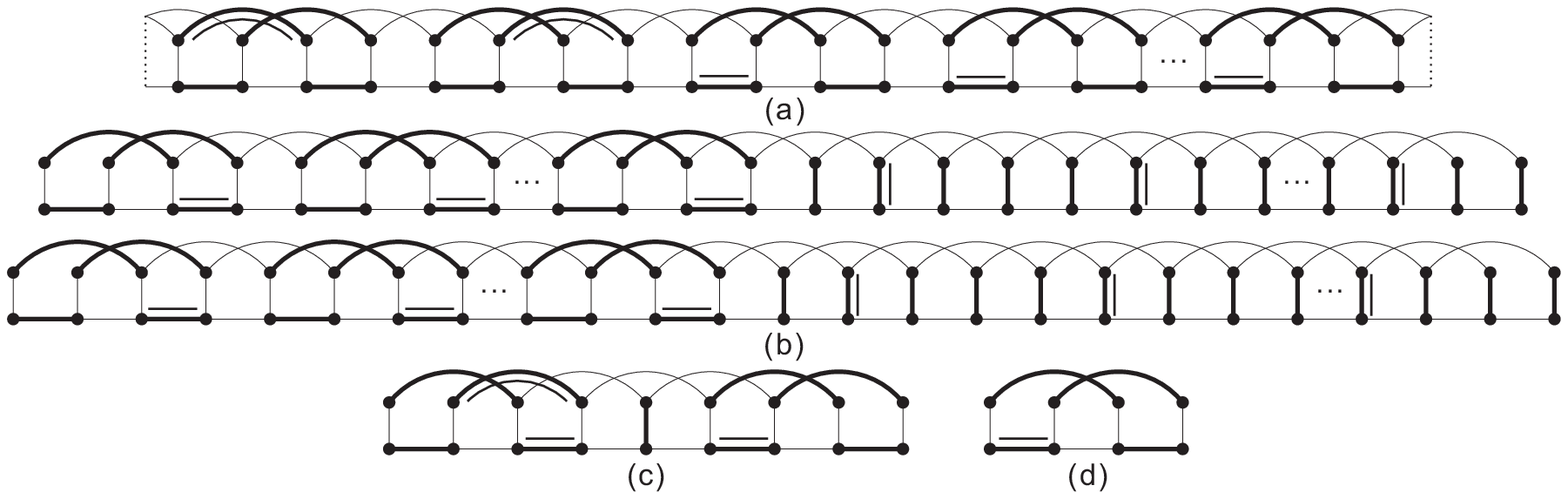}
            \caption{Illustration of Case 1 in the proof of Theorem \ref{thm1max}.}
            \label{algo1}
        \end{figure}

        \textbf{Case 2.} There is a segment $B$ with 2 or 3 (mod 4) $B$ in $P(n)$.

        We first pick an $AB$-chain with 2 or 3 (mod 4) $B$ and mark it with 1, then in turn mark the other such $AB$-chains with 2 or 3 (mod 4) $B$ alternatively with 0 and 1 from left to right. Namely, all such $AB$-chains are marked with $1010\cdots10$ if the total number is even, and marked with $1010\cdots 101$ otherwise. In turn for $j$-th $AB$-chain $W_j$, if it is either marked with 1, or unmarked with the immediate right-hand marked $AB$-chain marked with 1, then we give the edge subset $S_j$ in Fig. \ref{algocase}(a); otherwise, we give the edge subset $S_j$ in Fig. \ref{algocase}(b). Similar to the proof of Claim \ref{claim11}, we can confirm that $\cup_{j} S_j$ with cardinality no more than $\lceil \frac { n }{ 4 } \rceil$ is a forcing set of $M$.
    \end{proof}
    \begin{figure}[htbp]
        \centering
        \includegraphics[height=3.15in]{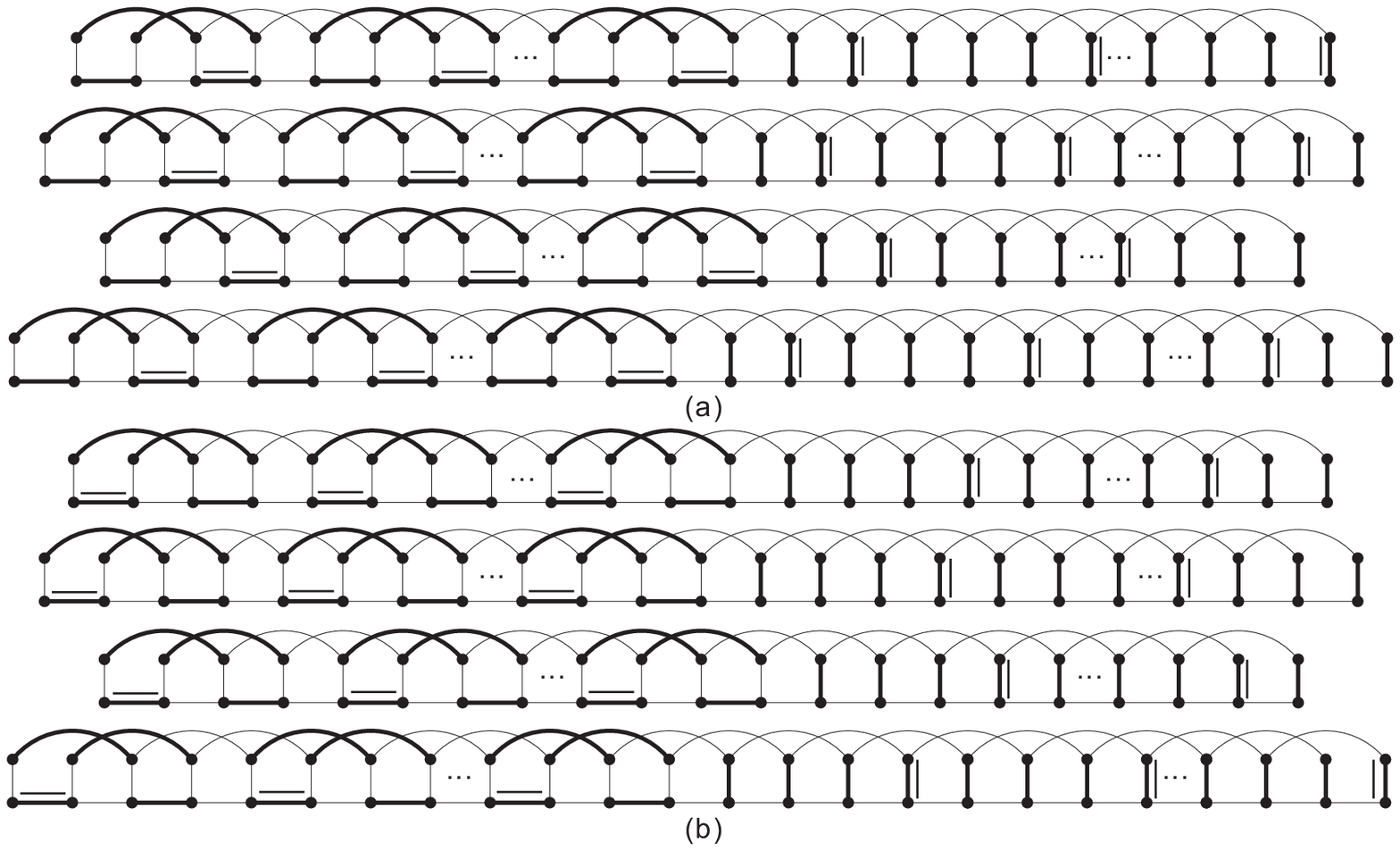}
        \caption{Illustration of Case 2 in the proof of Theorem \ref{thm1max}.}
        \label{algocase}
    \end{figure}

     Note that the above result dose not hold for $n=8$, and Theorem \ref{lemmima} fails in the non-bipartite graph $P(n)$. In fact, for a perfect matching $M$ in $\mathcal{M}_1(P(n))$ with $f(P(n),M)=\lceil \frac { n }{ 4 }\rceil$, we have $C(P(n),M)\le \lfloor \frac { n }{ 4 } \rfloor$, since the length of a shortest even cycle of $P(n)$ is 8.

\subsection{Minimum value of forcing numbers}
    \begin{thm}
        \label{thm1min}
        The minimum value of forcing numbers of perfect matchings in $\mathcal{M}_1(P(n))$ is $\lceil \frac { n+2 }{ 6 } \rceil$ for $n\ge 11.$
    \end{thm}
    \begin{proof}
        We divide our proof in two steps. First we find a perfect matching in $\mathcal{M}_1(P(n))$ with forcing number no more than $\lceil \frac { n+2 }{ 6 } \rceil$. Then we prove that the forcing number of each perfect matching in $\mathcal{M}_1(P(n))$ is no less than $\lceil \frac { n+2 }{ 6 } \rceil$.

        (1) Now we give a perfect matching $M$ of $P(n)$ expressed by
        \begin{align*}
            \left\{ \begin{array} {ll} BBBAA(BBBABA)^{\lfloor\frac {n-11}{12}\rfloor} & \makebox{if}~n\equiv 11~(\text{mod}~12), \\  BBBBAA(BBBABA)^{\lfloor\frac {n-11}{12}\rfloor} & \makebox{if}~n\equiv 0~(\text{mod}~12), \\ BBBBBAA(BBBABA)^{\lfloor\frac {n-11}{12}\rfloor} & \makebox{if}~n\equiv 1~(\text{mod}~12), \\ BBBBBBAA(BBBABA)^{\lfloor\frac {n-11}{12}\rfloor} & \makebox{if}~n\equiv 2~(\text{mod}~12), \\ BABABA(BBBABA)^{\lfloor\frac {n-11}{12}\rfloor} & \makebox{if}~n\equiv 3~(\text{mod}~12), \\ BBBBBBBABA(BBBABA)^{\lfloor\frac {n-11}{12}\rfloor} & \makebox{if}~n\equiv 4~(\text{mod}~12), \\ BBBBBAAA(BBBABA)^{\lfloor\frac {n-11}{12}\rfloor} & \makebox{if}~n\equiv 5~(\text{mod}~12), \\ BBBBBBAAA(BBBABA)^{\lfloor\frac {n-11}{12}\rfloor} & \makebox{if}~n\equiv 6~(\text{mod}~12), \\ AABBBAA(BBBABA)^{\lfloor\frac {n-11}{12}\rfloor} & \makebox{if}~n\equiv 7~(\text{mod}~12), \\  BAABBBAA(BBBABA)^{\lfloor\frac {n-11}{12}\rfloor} & \makebox{if}~n\equiv 8~(\text{mod}~12), \\ BBAABBBAA(BBBABA)^{\lfloor\frac {n-11}{12}\rfloor} & \makebox{if}~n\equiv 9~(\text{mod}~12), \\  BBBAABBBAA(BBBABA)^{\lfloor\frac {n-11}{12}\rfloor} & \makebox{if}~n\equiv 10~(\text{mod}~12).
            \end{array} \right.
        \end{align*}
        For the above sequences, we show the initial part in Fig. \ref{lower}(a), and the repeating part in Fig. \ref{lower}(b).

        \begin{figure}[ht]
            \centering
            \includegraphics[height=5.2in]{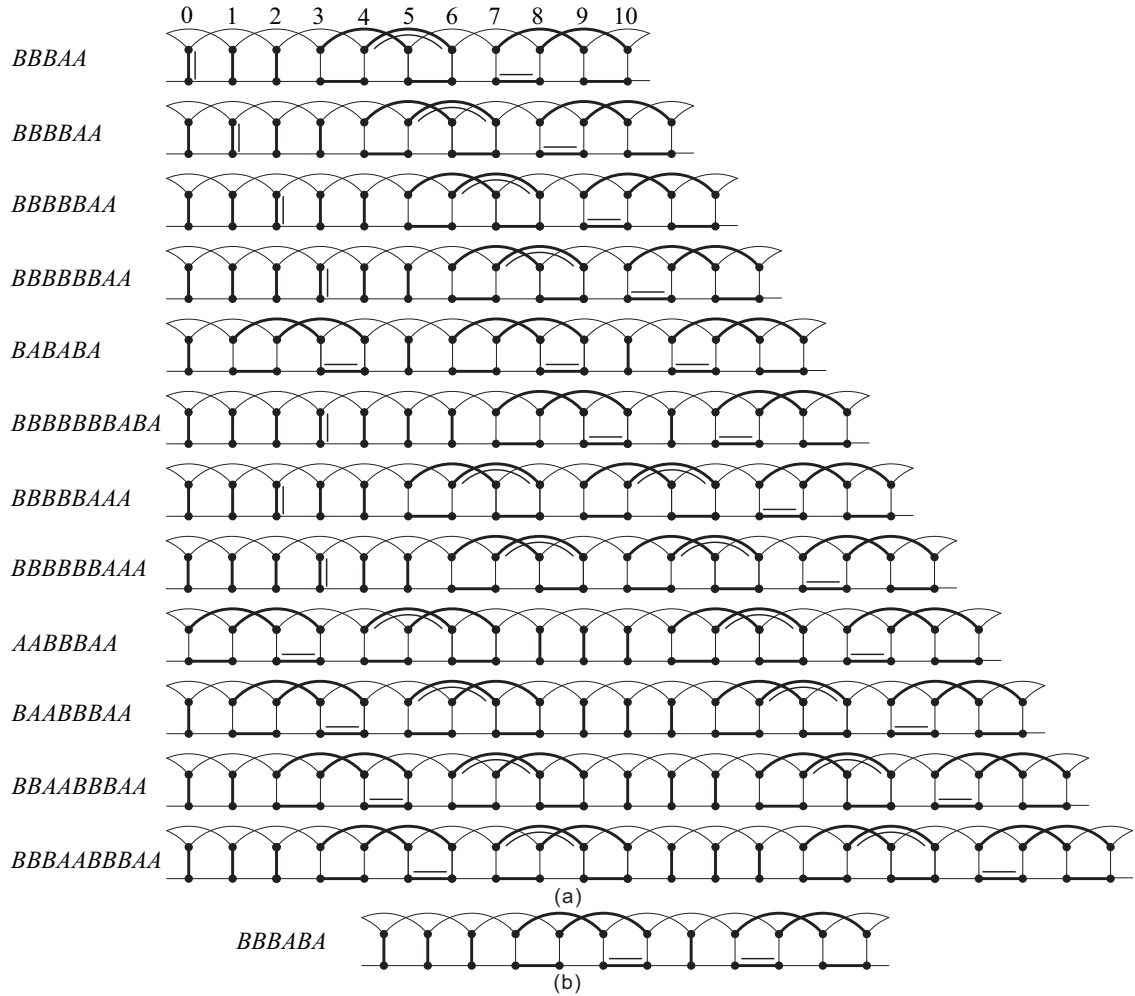}
            \caption{Perfect matching which achieves the lower bound of Theorem \ref{thm1min}.}
            \label{lower}
        \end{figure}

        \begin{cla}
            \label{claim22}
            Let $R$ be a subset of $V(P(n))$ with $u_i,v_{i+1},v_{i+5}\in R$ and $u_{i+2},v_{i+2},v_{i+3},u_{i+4},v_{i+4}$\\$\notin R.$ Then $u_{i+2}v_{i+2}$ belongs to all perfect matchings of $P(n)-R$ if there exists one$.$
        \end{cla}

        We illustrate the labels in Fig. \ref{lowerww}. Because $u_{i+2}$, $v_{i+2}$ and $v_{i+4}$ are odd components of $P(n)-R-u_{i+2}v_{i+2}-\{v_{i+3},u_{i+4}\}$, $P(n)-R-u_{i+2}v_{i+2}$ has no perfect matchings by Tutte's 1-factor Theorem. So the claim holds.

        \begin{figure}[htbp]
            \centering
            \includegraphics[height=0.45in]{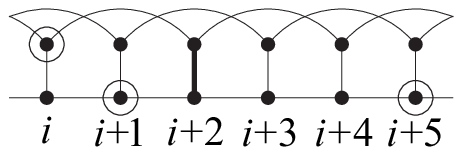}
            \caption{Illustration of Claim \ref{claim22} in the proof of Theorem \ref{thm1min}.}
            \label{lowerww}
        \end{figure}

        Let $S\subseteq M$ illustrated in Fig. \ref{lower} with cardinality $\lceil \frac { n+2 }{ 6 } \rceil$.

        \begin{cla}
            \label{claim33}
            $S$ is a forcing set of $M$ of $P(n).$
        \end{cla}

        We prove it by Corollary \ref{structure} (1) and (3) and Claim \ref{claim22}. Here we consider the case of $n\equiv 11$ (mod 12), and the other cases are similar. By a similar argument to Claim \ref{claim11} in Theorem \ref{thm1max}, we have $E_3=\{u_2v_2,v_5v_6,u_8u_{10}\}$, $E_5=E_3\cup\{u_1v_1,v_3v_4\}$, $E_6=E_5\cup\{u_3u_5\}$, $E_7=E_6\cup\{u_7u_9\}$, $E_8=E_7\cup\{v_9v_{10}\}$.

        If $n=11$, then $P(n)-V(S)-V(E_8)$ is empty. If $n=23$, then $E_9=E_8\cup\{u_{18}v_{18}\}$, $E_{11}=E_9\cup\{u_{14}u_{16},u_{20}u_{22}\}$, $E_{12}=E_{11}\cup\{u_{12}v_{12}\}$, $E_{13}=E_{12}\cup\{u_{11}v_{11}\}$. By Claim \ref{claim22}, $u_{13}v_{13}$ belongs to all perfect matchings of $P(n)-V(S)-V(E_{13})$. Then $E_{14}=E_{13}\cup\{u_{13}v_{13}\}$. For the resulting graph, $E_{15}=E_{14}\cup\{v_{14}v_{15}\}$, $E_{16}=E_{15}\cup\{u_{15}u_{17}\}$, $E_{17}=E_{16}\cup\{u_{19}u_{21}\}$, $E_{18}=E_{17}\cup\{v_{21}v_{22}\}$, and $P(n)-V(S)-V(E_{18})$ is empty. If $n\ge 35$, clearly we may continue to find new edges as stated in Corollary \ref{structure} until reaching an empty graph. Hence $S$ is a forcing set of $M$.

        (2) Next we prove that for each $M\in \mathcal{M}_1(P(n))$, we have $f(P(n),M)\ge \lceil \frac {n+2}{6} \rceil$. The initial cases of $11\le n\le 34$ can be verified from Table \ref{tab}. From now on suppose $n\ge 35$. To the contrary, suppose that $\mathcal{M}_1(P(n))$ has a perfect matching $M$ with a forcing set $S_0$ such that $|S_0|<\lceil \frac {n+2}{6} \rceil$. That is, $6|S_0|-n\le 1$. From the proof of Theorem \ref{thm1max}, we know that $M$ can be expressed by a sequence of at least one $A$ and at least one $B$.

        Let us consider $P(n)$ with perfect matching $M$ as follows. Given a chain decomposition $W_1,W_2,\ldots,W_m$ ($m\ge 1$), let
        \[\alpha (W_i)=6|S_0\cap E(W_i)|-|M\cap E(W_i)|\]
        for $i=1,2,\ldots,m$. Then
        \begin{align}
            \label{thmmin}
            \alpha (P(n)):=\sum_{i=1}^m\alpha (W_i)=6|S_0|-n\le 1.
        \end{align}

        For each $AB$-chain $W_i$, let $a_i$ and $b_i$ be the number of $A$ and $B$ in $W_i$, respectively. Then $a_i,b_i\ge 1$. Since chains $A$ and $BBBB$ each contains an $M$-alternating 8-cycle (see Fig. \ref{claim1}(b)) and thus at least one edge of $S_0$ by Theorem \ref{lem5}. This implies $|S_0\cap E(W_i)|\ge {a_i}+\lfloor\frac {b_i} 4 \rfloor$. If equality holds, then we say $W_i$ is a \emph{tight} $AB$-chain. Combining $|M\cap E(W_i)|=4a_i+b_i$, we have
        \begin{align}
            \label{eq2}
            \alpha (W_i)\ge 2a_i+6\lfloor\frac {b_i} 4 \rfloor-b_i.
        \end{align}

        \begin{cla}
            \label{star}
            If each chain $A$ in chain $BBABABB$ as $P(n)[\{u_j,v_j:j=i,i+1,\ldots,i+12\}]$ has precisely one edge in $S_0,$ then $v_{i+4}v_{i+5},v_{i+7}v_{i+8}\in S_0.$
        \end{cla}

        Suppose $v_{i+4}v_{i+5}\notin S_0$. Hence an $M$-alternating cycle $u_{i+1}u_{i+3}u_{i+5}v_{i+5}v_{i+4}v_{i+3}v_{i+2}v_{i+1}$\\$u_{i+1}$ contains no edges of $S_0$ if $u_{i+2}u_{i+4}\in S_0$, an $M$-alternating cycle $u_{i+2}u_{i+4}u_{i+6}v_{i+6}v_{i+5}$\\$v_{i+4}v_{i+3}v_{i+2}u_{i+2}$ contains no edges of $S_0$ if $u_{i+3}u_{i+5}\in S_0$, and an $M$-alternating cycle $u_iu_{i+2}u_{i+4}v_{i+4}v_{i+5}u_{i+5}u_{i+3}u_{i+1}v_{i+1}v_iu_i$ contains no edges of $S_0$ if $v_{i+2}v_{i+3}\in S_0$, a contradiction. Similarly, we have $v_{i+7}v_{i+8}\in S_0$.

        \begin{cla}
            \label{alter}
            There are no the following two continuous tight $AB$-chains$:$ $(1)$ $ABB$ and $ABB,$ $(2)$ $ABBB$ and $ABBB,$ $(3)$ $ABB$ and $ABBB,$ $(4)$ $ABBB$ and $ABB,$ $(5)$ $ABBB$ and $ABBBBBBB,$ $(6)$ $ABBBBBBB$ and $ABBB,$ $(7)$ $ABBBBBB$ and $ABBB,$ $(8)$ $AABBB$ and $ABBB,$ $(9)$ $AABB$ and $ABBB,$ $(10)$ $AABBB$ and $ABB.$
        \end{cla}

        Suppose there are two continuous tight $AB$-chains $ABB$ as $P(n)[\{u_j,v_j:j=i-4,i-3,\ldots,i+7\}]$. Hence an $M$-alternating cycle $u_{i+1}u_{i+3}u_{i+5}v_{i+5}v_{i+4}v_{i+3}v_{i+2}v_{i+1}u_{i+1}$ contains no edges of $S_0$ if $u_{i+2}u_{i+4}\in S_0$, an $M$-alternating cycle $u_{i+2}u_{i+4}u_{i+6}v_{i+6}v_{i+5}v_{i+4}v_{i+3}v_{i+2}$\\$u_{i+2}$ contains no edges of $S_0$ if $u_{i+3}u_{i+5}\in S_0$, and an $M$-alternating cycle $u_iu_{i+2}u_{i+4}u_{i+6}$\\$v_{i+6}v_{i+7}u_{i+7}u_{i+5}u_{i+3}u_{i+1}v_{i+1}v_iu_i$ contains no edges of $S_0$ if $v_{i+2}v_{i+3}$ or $v_{i+4}v_{i+5}\in S_0$, a contradiction. The other conclusions can be shown by a similar method.

        \begin{cla}
            There must exist a tight $AB$-chain $ABB$ or $ABBB$ in $P(n)$.
        \end{cla}

        For each $AB$-chain $W_i$, denote $b_i=4r_i+\varepsilon_i$ ($r_i\ge 0$, $0\le \varepsilon_i\le 3$). By Eq. (\ref{eq2}), we have
        \[\alpha (W_i)\ge 2a_i+2r_i-\varepsilon_i\ge -1.\]
        If $W_i$ is not tight, then $\alpha (W_i)\ge 2a_i+2r_i-\varepsilon_i+6\ge 5$.

        Furthermore, $\alpha (W_i)=-1$ if and only if $a_i=1$, $r_i=0$ and $\varepsilon_i=3$ ($W_i$ is tight chain $ABBB$); $\alpha (W_i)=0$ if and only if $a_i=1$, $r_i=0$ and $\varepsilon_i=2$ ($W_i$ is tight chain $ABB$); $\alpha (W_i)=1$ if and only if either $a_i=1$, $r_i=0$, $\varepsilon_i=1$ ($W_i$ is tight chain $AB$), $a_i=1$, $r_i=1$, $\varepsilon_i=3$ ($W_i$ is tight chain $ABBBBBBB$), or $a_i=2$, $r_i=0$, $\varepsilon_i=3$ ($W_i$ is tight chain $AABBB$); $\alpha (W_i)=2$ if and only if either $a_i=1$, $r_i=1$, $\varepsilon_i=2$ ($W_i$ is tight chain $ABBBBBB$), or $a_i=2$, $r_i=0$, $\varepsilon_i=2$ ($W_i$ is tight chain $AABB$). It follows that if there are no tight $AB$-chains $ABB$ or $ABBB$ in $P(n)$, then $\alpha (P(n))\ge 2$ by $n\ge 35$, a contradiction to Eq. (\ref{thmmin}).

        Suppose there are no tight $AB$-chains $ABBB$. In turn we denote the tight $AB$-chains $ABB$ by $U_1,U_2,\ldots,U_l$ ($l\ge 1$). Then for the chain $V_i$ between two consecutive tight $AB$-chains $ABB$ $U_i$ and $U_{i+1}$ (the subscripts module $l$), we have $\alpha(V_i)\ge 1$ by Claim \ref{alter} (1), and equality holds if and only if $V_i$ is tight $AB$-chain $AB$, $ABBBBBBB$ or $AABBB$. It follows by $n\ge 35$ that $\alpha (P(n))=\sum_{i=1}^l \alpha(V_i)\ge 2$, a contradiction to Eq. (\ref{thmmin}).

        Suppose there are $l(\ge 1)$ tight $AB$-chains $ABBB$. In turn we denote the tight $AB$-chains $ABBB$ by $U_1,U_2,\ldots,U_l$. Then for the chain $V_i$ between two consecutive tight $AB$-chains $ABBB$ $U_i$ and $U_{i+1}$ (the subscripts module $l$), we have $\alpha(V_i)\ge 1$ by Claim \ref{alter} (2-4) and $\alpha(V_i)\le 2$ by Eq. (\ref{thmmin}). Hence we can consider the following two cases.

        \textbf{Case 1.} $\alpha(V_i)= 1$ for each $1\le i\le l$. Then each $V_i$ is tight $AB$-chain $AB$, $ABBBBBBB$ or $AABBB$. From Claim \ref{alter} (6) and (8), it follows that $V_i$ is chain $AB$. Then $M=(ABABBB)^{\frac n {12}}$ and every $AB$-chain is tight. By Claim \ref{star}, we can completely determine $S_0$ illustrated in Fig. \ref{namely}(a). However, an $M$-alternating cycle in Fig. \ref{namely}(a) contains no edges of $S_0$, a contradiction.

        \textbf{Case 2.} There is a chain $V_j$ with $\alpha(V_j)=2$. Then for all the others $V_i$ ($i\neq j$), we have $\alpha(V_i)= 1$ by Eq. (\ref{thmmin}). Furthermore, the number of $AB$-chains that contained in $V_j$ is more than one by Claim \ref{alter} (7) and (9), and less than four by Claim \ref{alter} (1), (3) and (4). From Claim \ref{alter} (3-6), (8) and (10), it follows that each $V_i$ is chain $AB$ and $V_j$ should be precisely one of the following cases.

        \textbf{Case 2.1.} $V_j$ is chain $ABABBAB$. Then $M=ABABB(ABABBB)^{\frac{n-11}{12}}$ and each $AB$-chain is tight. By Claim \ref{star}, we can completely determine $S_0$ illustrated in Fig. \ref{namely}(b). Hence an $M$-alternating cycle in Fig. \ref{namely}(b) contains no edges of $S_0$, a contradiction.

        \textbf{Case 2.2.} $V_j$ is chain $ABAB$. Then $M=AB(ABABBB)^{\frac{n-5}{12}}$ and each $AB$-chain is tight. Similar to the proof of Claim \ref{star}, we can confirm that the edge of the first, second and third chains $A$ in $S_0$ are $v_2v_3$, $v_5v_6$ or $v_7v_8$, and $v_{10}v_{11}$, respectively. By Claim \ref{star}, we can determine $S_0$. Hence an $M$-alternating cycle $C_{2,2}$ in Fig. \ref{namely}(c) contains no edges of $S_0$ if $v_5v_6\in S_0$, and an $M$-alternating cycle obtained from $C_{2,2}$ by switching $u_{n-5}v_{n-5}v_{n-4}v_{n-3}\cdots v_{15}v_{16}u_{16}u_{18}$ to $u_{n-5}u_{n-3}v_{n-3}v_{n-2}u_{n-2}u_0u_2u_4v_4v_5v_6u_6u_8u_{10}u_{12}v_{12}v_{13}$\\$v_{14}u_{14}u_{16}v_{16}v_{17}v_{18}u_{18}$ contains no edges of $S_0$ if $v_7v_8\in S_0$, a contradiction.

        \textbf{Case 2.3.} $V_j$ is chain $AABBBAB$. Then $M=AABBB(ABABBB)^{\frac{n-11}{12}}$ and each $AB$-chain is tight. Similar to the proof of Claim \ref{star}, we can confirm that the first chain $AA$ has two edges either $v_2v_3$ and $u_4u_6$, $u_1u_3$ and $v_4v_5$, or $u_1u_3$ and $u_4u_6$ in $S_0$. By Claim \ref{star}, we can determine $S_0$. Hence an $M$-alternating cycle $C_{2,3}$ in Fig. \ref{namely}(d) contains no edges of $S_0$ if $v_2v_3,u_4u_6\in S_0$, an $M$-alternating cycle obtained from $C_{2,3}$ by switching $v_{n-1}v_{0}v_1u_1\cdots v_9v_{10}u_{10}u_{12}$ to $v_{n-1}v_{n-2}u_{n-2}u_0u_2u_4u_6v_6v_7v_8u_8u_{10}v_{10}v_{11}v_{12}u_{12}$ contains no edges of $S_0$ if $u_1u_3,v_4v_5\in S_0$, and an $M$-alternating cycle obtained from $C_{2,3}$ by switching $v_1u_1u_3u_5\cdots v_9v_{10}u_{10}u_{12}$ to $v_1v_2v_3v_4v_5v_6v_7v_8u_8u_{10}v_{10}v_{11}v_{12}u_{12}$ contains no edges of $S_0$ if $u_1u_3,u_4u_6\in S_0$, a contradiction.
    \end{proof}
    \begin{figure}[htbp]
        \centering
        \includegraphics[height=1.47in]{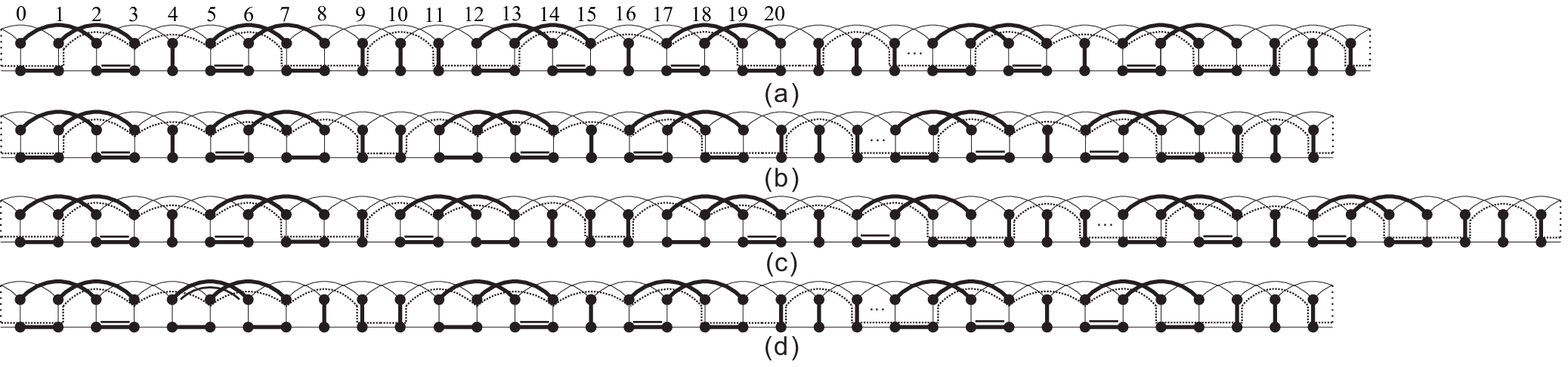}
        \caption{Illustration of Cases 1 and 2 in the proof of Theorem \ref{thm1min}.}
        \label{namely}
    \end{figure}

    Note that the above result does not hold for $n=10$.

\subsection{Continuity}
    \begin{thm}
        \label{thm1con}
        For $n\ge 3,$ $\{f(P(n),M): M \in \mathcal{M}_1(P(n))\}$ is continuous$.$
    \end{thm}
    \begin{proof}
        Let $M_1$ be a perfect matching in $\mathcal{M}_1(P(n))$ expressed by a sequence of at least one $A$ and at least one $B$, $M_2$ be the perfect matching obtained from $M_1$ by transforming one chain $BA$ to $B^5$, and maintaining the other parts. We illustrate the labels in Fig. \ref{con1}. In fact, $M_2$ is the symmetric difference between $M_1$ and the $M_1$-alternating cycle $u_{i+1}u_{i+3}v_{i+3}v_{i+4}u_{i+4}u_{i+2}v_{i+2}v_{i+1}u_{i+1}$.

        \begin{figure}[htbp]
            \centering
            \includegraphics[height=0.45in]{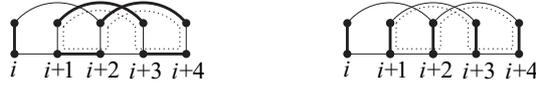}
            \caption{Alternating cycles in chains $BA$ and $B^5$.}
            \label{con1}
        \end{figure}

        Denote the subgraph $P(n)[\{u_j,v_j:j=i,i+1,i+2,i+3,i+4\}]$ by $W$. Then for each minimum forcing set $S_1$ of $M_1$, the number of edges of $W$ in $S_1$ is no less than one by Theorem \ref{lem5}, and no more than two. This is because $S_1\setminus E(W)$ is a forcing set of $M_1\setminus E(W)$ of the subgraph $P(n)-V(W)$; every $M_1$-alternating cycle that does not contain any edge of $M_1\cap E(W)$ must be contained in $P(n)-V(W)$; the edges $u_iv_i$ and $u_{i+1}u_{i+3}$ can \emph{determine} all edges of $M_1\cap E(W)$ (i.e. every $M_1$-alternating cycle which contains some edge in $M_1\cap E(W)$ must contain edge $u_iv_i$ or $u_{i+1}u_{i+3}$). By Theorem \ref{lem5}, we have that $S_1\setminus E(W)\cup \{u_iv_i,u_{i+1}u_{i+3}\}$ is a forcing set of $M_1$. From $S_1$, we could obtain a forcing set of $M_2$ by transforming all edges in $S_1\cap E(W)$ to two edges $u_iv_{i},$ $u_{i+4}v_{i+4}$ and maintaining the other edges, which implies $f(P(n),M_2)\le f(P(n),M_1)+1$. Similarly, we could obtain $f(P(n),M_1)\le f(P(n),M_2)+1$.

        Except for the perfect matching $A^{\frac n 4}$, we can give a series of transformations similar as above from each perfect matching in $\mathcal{M}_1(P(n))$ to $B^n$ with the variation of forcing numbers during each transformation no more than one. For the special case, $f(P(n),A^{\frac n 4})-f(P(n),B^n)=1$ if $n=8$, and $f(P(n),A^{\frac n 4})-f(P(n),B^n)=0$ otherwise. Then the theorem holds.
    \end{proof}

\section{Forcing number of a perfect matching in $\mathcal{M}_2(P(n))$}
    In this section, we first derive the maximum and minimum forcing numbers of second type of perfect matchings, then prove the continuity. In detail, for $n\ge 34$, the set of forcing numbers of second type of perfect matchings form $\left[\lceil \frac { n }{ 12 }  \rceil+1,\lceil \frac { n+3 }{ 7 }  \rceil +\delta (n)\right]$.

\subsection{Maximum value of forcing numbers}
    \begin{thm}
        \label{thm2max}
        The maximum value of forcing numbers of perfect matchings in $\mathcal{M}_2(P(n))$ is $\lceil \frac { n+3 }{ 7 } \rceil +\delta (n)$ for $n\ge 34,$ where $\delta (n)=1$ if $n\equiv 3$ \emph{(mod 7),} and $\delta (n)=0$ otherwise$.$
    \end{thm}
    \begin{proof}
        We divide our proof in two steps. First we find a perfect matching in $\mathcal{M}_2(P(n))$ with forcing number no less than $\lceil \frac { n+3 }{ 7 } \rceil +\delta (n)$. Then we prove that the forcing number of each perfect matching in $\mathcal{M}_2(P(n))$ is no more than $\lceil \frac { n+3 }{ 7 } \rceil +\delta (n)$.

        (1) Now we give a perfect matching $M$ (see Fig. \ref{upper}) of $P(n)$ expressed by
        \begin{align*}
            \left\{ \begin{array} {ll} CCCD(CD)^{\lfloor \frac{n-13} 7\rfloor} & \makebox{if}~n\equiv 6~(\text{mod}~7), \\ CDCD(CD)^{\lfloor \frac{n-13} 7\rfloor} & \makebox{if}~n\equiv 0~(\text{mod}~7), \\ DDCD(CD)^{\lfloor \frac{n-13} 7\rfloor} & \makebox{if}~n\equiv 1~(\text{mod}~7), \\ CCCCD(CD)^{\lfloor \frac{n-13} 7\rfloor} & \makebox{if}~n\equiv 2~(\text{mod}~7), \\  CCDCD(CD)^{\lfloor \frac{n-13} 7\rfloor} & \makebox{if}~n\equiv 3~(\text{mod}~7), \\ DCDCD(CD)^{\lfloor \frac{n-13} 7\rfloor} & \makebox{if}~n\equiv 4~(\text{mod}~7), \\ CDCCDCCDCCDC(DC)^{\frac{n-40} 7} & \makebox{if}~n\equiv 5~(\text{mod}~7). \end{array} \right.
        \end{align*}

        \begin{figure}[htbp]
            \centering
            \includegraphics[height=3in]{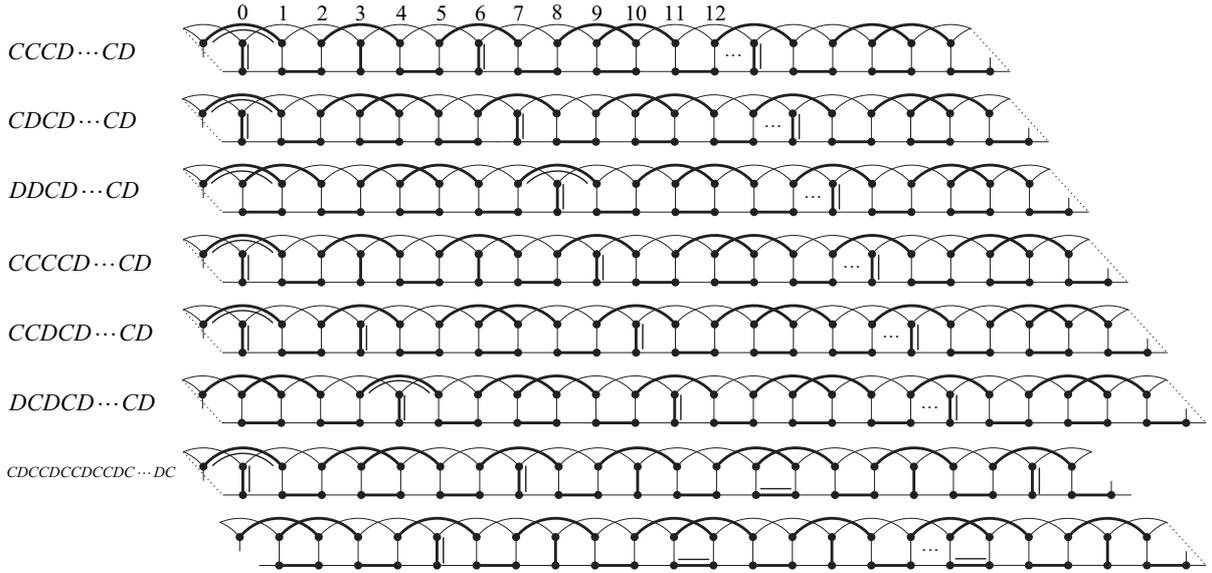}
            \caption{Perfect matching which achieves the upper bound of Theorem \ref{thm2max}.}
            \label{upper}
        \end{figure}

        \begin{cla2}
            \label{quan}
            Given a perfect matching $M$ of $P(n)$ with a forcing set $S,$ chains $CD$ and $DC$ each contains an $M$-alternating 8-cycle (see Fig. \ref{upper11}) and thus at least one edge of $S.$
        \end{cla2}

        \begin{figure}[htbp]
            \centering
            \includegraphics[height=0.45in]{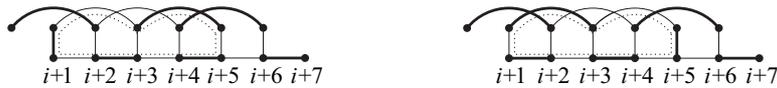}
            \caption{$M$-Alternating 8-cycles in chains $CD$ and $DC$.}
            \label{upper11}
        \end{figure}

        Next we prove $f(P(n),M)\ge \lceil \frac { n+3 }{ 7 } \rceil +\delta (n)$. To the contrary, suppose there is a forcing set $S_0$ of $M$ such that $|S_0|<\left\lceil \frac { n+3 }{ 7 }  \right\rceil +\delta (n)$.

        \textbf{Case 1.} $n\equiv 0,4$ (mod 7). Here we consider the case of $n\equiv 4$ (mod 7), and the other case is similar. Let $n=7k+4$ $(k\ge 5)$. Then $\left|S_0\right|\le k$ and $M=D(CD)^k$. So $\left|S_0\right|\ge k$ by Claim \ref{quan}. Hence $|S_0|=k$ and $S_0$ consists of precisely one edge of each chain $CD$. Since the first chain $DCD$ contains precisely one edge of $S_0$, the edge must be spoke $u_4v_4$ by Claim \ref{quan}. In general, it follows that $S_0$ consists of precisely one edge of each chain $C$ and $S_0=\{u_{7i+4}v_{7i+4}:i=0,1,\ldots,k-1\}$ (see Fig. \ref{upperkk}(a)). Hence an $M$-alternating cycle in Fig. \ref{upperkk}(a) contains no edges of $S_0$, a contradiction.

        \textbf{Case 2.} $n\equiv 1,2,3,6$ (mod 7). Here we consider the case of $n\equiv 2$ (mod 7), and the other cases are similar. Let $n=7k+9$ $(k\ge 4)$. Then $\left|S_0\right|\le k+1$ and $M=CCC(CD)^k$. So $|S_0|\ge k$ by Claim \ref{quan}. Let chain $V=CCC$ and $W=(CD)^k$. Then $M=VW$. We use the notations $\underline{C}$ and $\underline{D}$ in some sequence to denote the fact: such chain contains precisely one edge of $S_0$.

        If $|S_0|=k$, then each chain $CD$ contains precisely one edge of $S_0$. Note that $CDC$ is a chain in that cyclic sequence. It follows that $S_0$ satisfies $CCC(C\underline{D})^k$ and $S_0=\{v_{7i+12}v_{7i+13}:i=0,1,\ldots,k-1\}$ (see Fig. \ref{upperkk}(b)). Hence an $M$-alternating cycle $C_{2}$ in Fig. \ref{upperkk}(b) contains no edges of $S_0$, a contradiction. So $|S_0|=k+1$.

        \begin{figure}[htbp]
            \centering
            \includegraphics[height=1in]{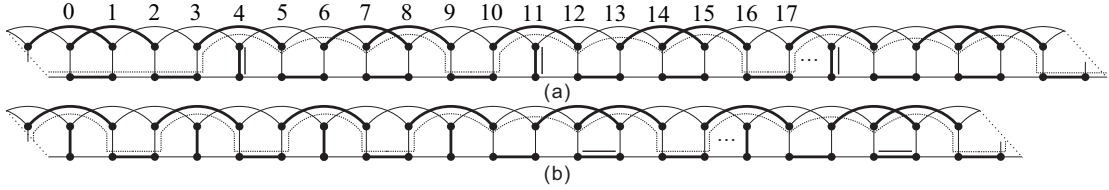}
            \caption{Illustration of Cases 1 and 2 in the proof of Theorem \ref{thm2max}.}
            \label{upperkk}
        \end{figure}

        Suppose that $S_0$ contains one edge of $V$. If $S_0$ satisfies $\underline {C}CC(C\underline {D})^{k_1}(\underline {C}D)^{k_2}$ ($k_1+k_2=k$, $k_1\ge 0$, $k_2\ge 1$), then $S_0\subset \{u_0v_0,u_{7i+9}v_{7i+9},v_{7i+10}v_{7i+11},u_{7i+11}u_{7i+13},v_{7i+12}v_{7i+13}:i=0,1,\ldots,k-1\}$. Hence the above $M$-alternating cycle $C_{2}$ also contains no edges of $S_0$, a contradiction. Otherwise, $S_0$ contains precisely one edge of each chain $D$ of $W$ and $S_0\cap E(W)\subset \{u_{7i+11}u_{7i+13},v_{7i+12}v_{7i+13}:i=0,1,\ldots,k-1\}$. Hence an $M$-alternating cycle obtained from $C_{2}$ by switching $v_{n-1}u_{n-1}u_1v_1\cdots u_7v_7v_8u_8$ to $v_{n-1}v_0u_0u_2u_4v_4v_5v_6u_6u_8$ contains no edges of $S_0$ if $V$ has an edge $u_{n-1}u_1,$ $v_1v_2,$ $u_3v_3,$ $u_5u_7$ or $v_7v_8$ in $S_0$, and an $M$-alternating cycle obtained from $C_{2}$ by switching $v_2u_2u_4v_4v_5u_5$ to $v_2v_3u_3u_5$ contains no edges of $S_0$ if $V$ has an edge $u_0v_0,$ $u_2u_4,$ $v_4v_5$ or $u_6v_6$ in $S_0$, a contradiction.

        Suppose that $V$ contains no edges of $S_0$. Then one chain $CD$ in $W$, say $U$, contains precisely two edges of $S_0$ and the other chains $CD$ each contains precisely one edge of $S_0$. Let $U=P(n)[\{u_j,v_{j+1}:j=7i+8,7i+9,\ldots,7i+14\}]$ $(0\le i\le k-1)$.

        \textbf{Case 2.1.} $S_0$ satisfies $CCC(C\underline {D})^{k_1}(\underline {C}D)^{k_2}\underline {CD}(C\underline {D})^{k_3}$, where $k_1+k_2+k_3=k-1$ and $k_1,k_2,k_3\ge 0$. Then $S_0\setminus E(U) \subset \{u_{7j+9}v_{7j+9},v_{7j+10}v_{7j+11},u_{7j+11}u_{7j+13},v_{7j+12}v_{7j+13}:i\neq j=0,1,\ldots,k-1\}$. Hence an $M$-alternating cycle obtained from $C_{2}$ by switching $u_{7i+3}u_{7i+5}u_{7i+7}v_{7i+7}v_{7i+8}u_{7i+8}$ to $u_{7i+3}v_{7i+3}v_{7i+4}u_{7i+4}u_{7i+6}u_{7i+8}$ contains no edges of $S_0$ if $C$ in $U$ has an edge $u_{7i+2}v_{7i+2}$ in $S_0$, and an $M$-alternating cycle obtained from $C_{2}$ by switching $v_{7i+1}u_{i+1}u_{7i+3}u_{7i+5}u_{7i+7}v_{7i+7}v_{7i+8}u_{7i+8}$ to $v_{7i+1}v_{7i+2}u_{7i+2}u_{7i+4}u_{7i+6}u_{7i+8}$ contains no edges of $S_0$ if $C$ in $U$ has an edge $u_{7i+1}u_{7i+3}$ or $v_{7i+3}v_{7i+4}$ in $S_0$, a contradiction.

        \textbf{Case 2.2.} $C$ or $D$ in $U$ contains two edges of $S_0$. It follows that the edges in $S_0$ are all contained in chains $D$ and $S_0\setminus E(U) \subset \{u_{7j+11}u_{7j+13},v_{7j+12}v_{7j+13}:i\neq j=0,1,\ldots,k-1\}$. Given an $M$-alternating cycle $C_2'$ which obtained from $C_{2}$ by switching $u_{7i+7}v_{7i+7}v_{7i+8}u_{7i+8}\cdots u_{7i+14}v_{7i+14}v_{7i+15}u_{7i+15}$ to $u_{7i+7}u_{7i+9}v_{7i+9}v_{7i+10}v_{7i+11}u_{7i+11}u_{7i+13}$\\$u_{7i+15}$ if both edges of $U$ in $S_0$ are from $\{u_{7i+4}u_{7i+6},v_{7i+5}v_{7i+6},v_{7i+7}v_{7i+8}\}$ and $i< k$, by switching $u_{n-2}v_{n-2}v_{n-1}u_{n-1}u_1v_1$ to $u_{n-2}u_0v_0v_1$ if both edges of $U$ in $S_0$ are from $\{u_{7i+4}u_{7i+6},v_{7i+5}v_{7i+6},v_{7i+7}v_{7i+8}\}$ and $i= k$, by switching $u_{7i+3}u_{7i+5}u_{7i+7}v_{7i+7}v_{7i+8}u_{7i+8}$ to $u_{7i+3}v_{7i+3}v_{7i+4}u_{7i+4}u_{7i+6}u_{7i+8}$ if $U$ has two edges $u_{7i+5}u_{7i+7}$ and $v_{7i+5}v_{7i+6}$, or $u_{7i+5}u_{7i+7}$ and $v_{7i+7}v_{7i+8}$ in $S_0$, and by switching $u_{7i+3}u_{7i+5}u_{7i+7}v_{7i+7}$ to $u_{7i+3}v_{7i+3}v_{7i+4}v_{7i+5}v_{7i+6}v_{7i+7}$ if $U$ has two edges $u_{7i+4}u_{7i+6}$ and $u_{7i+5}u_{7i+7}$ in $S_0$. In all cases mentioned above, the obtained $M$-alternating cycle $C_2'$ contains no edges of $S_0$, a contradiction.

        \textbf{Case 3.} $n\equiv 5$ (mod 7). Let $n=7k+40$ $(k\ge 0)$. Then $\left|S_0\right|\le k+6$ and $M=(CDC)^4(DC)^k$. Let chain $V=(CDC)^4$ and $W=(DC)^k$. Then $M=VW$. By Claim \ref{quan}, we have $4\le |S_0\cap E(V)|\le 6$.

        Suppose $|S_0\cap E(V)|=4$. Then $S_0\cap E(V)=\{v_{10i+3}v_{10i+4}:i=0,1,2,3\}$. Hence an $M$-alternating cycle $u_{7}u_{9}u_{11}u_{13}u_{15}u_{17}v_{17}v_{18}v_{19}v_{20}u_{20}u_{18}u_{16}u_{14}u_{12}u_{10}v_{10}v_{9}v_{8}v_{7}u_{7}$ contains no edges of $S_0$, and an $M$-alternating cycle $u_{17}u_{19}u_{21}u_{23}u_{25}u_{27}v_{27}v_{28}v_{29}v_{30}u_{30}u_{28}u_{26}u_{24}u_{22}u_{20}$\\$v_{20}v_{19}v_{18}v_{17}u_{17}$ contains no edges of $S_0$, a contradiction.

        If $|S_0\cap E(V)|=5$, then one chain $CDC$ in $V$, say $U$, contains precisely two edges of $S_0$, and the other chains $CDC$ each contains precisely one edge of $S_0$. By a similar argument as above, we have $k\ge 1$ and $U$ could not be the first or last one. Suppose $|S_0|=k+5$ and $U$ is the second one. Hence $S_0$ contains precisely one edge of each chain $D$ of $W$. Note that $CDC$ is a chain in that cyclic sequence. So $S_0\setminus E(U)\subset \{v_{10i+3}v_{10i+4},v_{7i+40}v_{7i+41}:i=0,1,2,3,j=0,1,\ldots,k-1\}$. Since there is an $M$-alternating cycle $C_{3}$ in Fig. \ref{upperkk11}(a) and an $M$-alternating cycle obtained from $C_{3}$ by switching $v_9u_9u_{11}u_{13}\cdots u_{25}u_{27}v_{27}v_{28}$ to $v_9 v_{10}u_{10}u_{12}u_{14}u_{16}u_{18}u_{20}v_{20}v_{21}v_{22}u_{22}u_{24}u_{26}u_{28}v_{28}$, we have $\{u_9u_{11},u_{13}u_{15},u_{17}v_{17},v_{18}v_{19}\}\cap S_0\neq \emptyset$ and $\{u_{10}v_{10},u_{12}u_{14},u_{16}u_{18}\}\cap S_0\neq \emptyset$ by Theorem \ref{lem5}. Combining Claim \ref{quan}, we know the two edges of $U$ in $S_0$ are one from $\{u_{13}u_{15},u_{17}v_{17}\}$ and the other from $\{u_{10}v_{10},u_{12}u_{14}\}$. Hence an $M$-alternating cycle obtained from $C_{3}$ by switching $u_{11}u_{13}u_{15}u_{17}\cdots u_{25}u_{27}v_{27}v_{28}$ to $u_{11}v_{11}v_{12}v_{13}v_{14}v_{15}v_{16}u_{16}u_{18}u_{20}v_{20}v_{21}v_{22}u_{22}u_{24}u_{26}u_{28}v_{28}$ contains no edges of $S_0$, a contradiction. Similarly, if either $|S_0|=k+5$ and $U$ is the third one or $|S_0|=k+6$, then it also deduces a contradiction. So $|S_0\cap E(V)|=6$ and $S_0$ contains precisely one edge of each chain $DC$ of $W$.

        \textbf{Case 3.1.} Two chains $CDC$ in $V$, say $U_1$ and $U_2$, each contains precisely two edges of $S_0$ and the other two each contains precisely one edge of $S_0$. Since there is an $M$-alternating path $u_{10i-1}u_{10i+1}u_{10i+3}u_{10i+5}v_{10i+5}v_{10i+6}u_{10i+6}u_{10i+8}$ containing no edge $u_{10i}v_{10i},$ $v_{10i+1}v_{10i+2},$ $u_{10i+2}u_{10i+4},$ $v_{10i+3}v_{10i+4}$ or $u_{10i+7}v_{10i+7}$, an $M$-alternating path $u_{10i-1}$\\$u_{10i+1}u_{10i+3}u_{10i+5}u_{10i+7}v_{10i+7}v_{10i+8}v_{10i+9}$ containing no edge $v_{10i+1}v_{10i+2},$ $u_{10i+2}u_{10i+4},$\\$v_{10i+3}v_{10i+4},$ $v_{10i+5}v_{10i+6}$ or $u_{10i+6}u_{10i+8}$, an $M$-alternating path $u_{10i-1}u_{10i+1}v_{10i+1}v_{10i+2}$\\$u_{10i+2}u_{10i+4}u_{10i+6}u_{10i+8}$ containing no edge $u_{10i}v_{10i},$ $u_{10i+3}u_{10i+5},$ $v_{10i+3}v_{10i+4},$ $v_{10i+5}v_{10i+6}$ or $u_{10i+7}v_{10i+7}$, an $M$-alternating path $u_{10i-1}u_{10i+1}v_{10i+1}v_{10i+2}v_{10i+3}v_{10i+4}v_{10i+5}v_{10i+6}u_{10i+6}$\\$u_{10i+8}$ containing no edge $u_{10i+2}u_{10i+4}$ or $u_{10i+3}u_{10i+5}$, an $M$-alternating path $u_{10i-1}u_{10i+1}$\\$u_{10i+3}u_{10i+5}v_{10i+5}v_{10i+6}v_{10i+7}u_{10i+7}u_{10i+9}u_{10i+11}$ containing no edge $v_{10i+3}v_{10i+4}$ or $v_{10i+8}$\\$v_{10i+9}$, an $M$-alternating path $v_{10i-2}v_{10i-1}v_{10i}u_{10i}u_{10i+2}u_{10i+4}u_{10i+6}u_{10i+8}$ containing no edge $u_{10i-1}u_{10i+1},$ $v_{10i+1}v_{10i+2},$ $u_{10i+3}u_{10i+5}$ or $v_{10i+3}v_{10i+4}$, and an $M$-alternating path $u_{7j+32}u_{7j+34}u_{7j+36}u_{7j+38}v_{7j+38}v_{7j+39}$ containing no edge $v_{7j+40}v_{7j+41}$ for $i=0,1,2,3$ and $j=0,1,\ldots,k-1$, we can generate an $M$-alternating cycle which contains no edges of $S_0$ except for the cases that the four edges of $U_1$ and $U_2$ are either $v_{10s+3}v_{10s+4},$ $v_{10s+8}v_{10s+9},$ $u_{10s+9}u_{10s+11},$ $v_{10s+13}v_{10s+14}$, or $v_{10s+3}v_{10s+4},$ $v_{10s+8}v_{10s+9},$ $v_{10s+11}v_{10s+12},u_{10s+13}u_{10s+15}$ for $s=0,1,2$. However, we can transform the above exceptions to the following Case 3.2 by changing the four edges into $u_{10s-1}u_{10s+1},$ $u_{10s}v_{10s},$ $u_{10s+7}v_{10s+7},v_{10s+13}v_{10s+14}$ and maintaining the other edges in $S_0$. Namely, if for the above exceptions $S_0$ is a forcing set of $M$, then for the transformation case $S_0$ is also a forcing set of $M$.

        \textbf{Case 3.2.} One chain $CDC$ in $V$, say $X$, contains precisely three edges of $S_0$, and the other chains $CDC$ each contains precisely one edge of $S_0$. Similar as above, we have $k\ge 1$ and $X$ could not be the first or last one. Then $S_0\setminus E(X)\subset \{v_{10i+3}v_{10i+4},v_{7i+40}v_{7i+41}:i=0,1,2,3,j=0,1,\ldots,k-1\}$. Hence an $M$-alternating cycle in Fig. \ref{upperkk11}(b) contains no edges of $S_0$ if $k$ is odd and $X$ is the second one, an $M$-alternating cycle in Fig. \ref{upperkk11}(c) contains no edges of $S_0$ if $k$ is even and $X$ is the second one, and an $M$-alternating cycle $u_1u_{3}u_{5}u_{7}v_{7}v_{8}v_{9}v_{10}u_{10}u_{8}u_{6}u_{4}u_{2}u_{0}v_{0}v_{n-1}v_{n-2}v_{n-3}u_{n-3}u_{n-1}u_{1}$ contains no edges of $S_0$ if $X$ is the third one, a contradiction.

        \begin{figure}[htbp]
            \centering
            \includegraphics[height=2.65in]{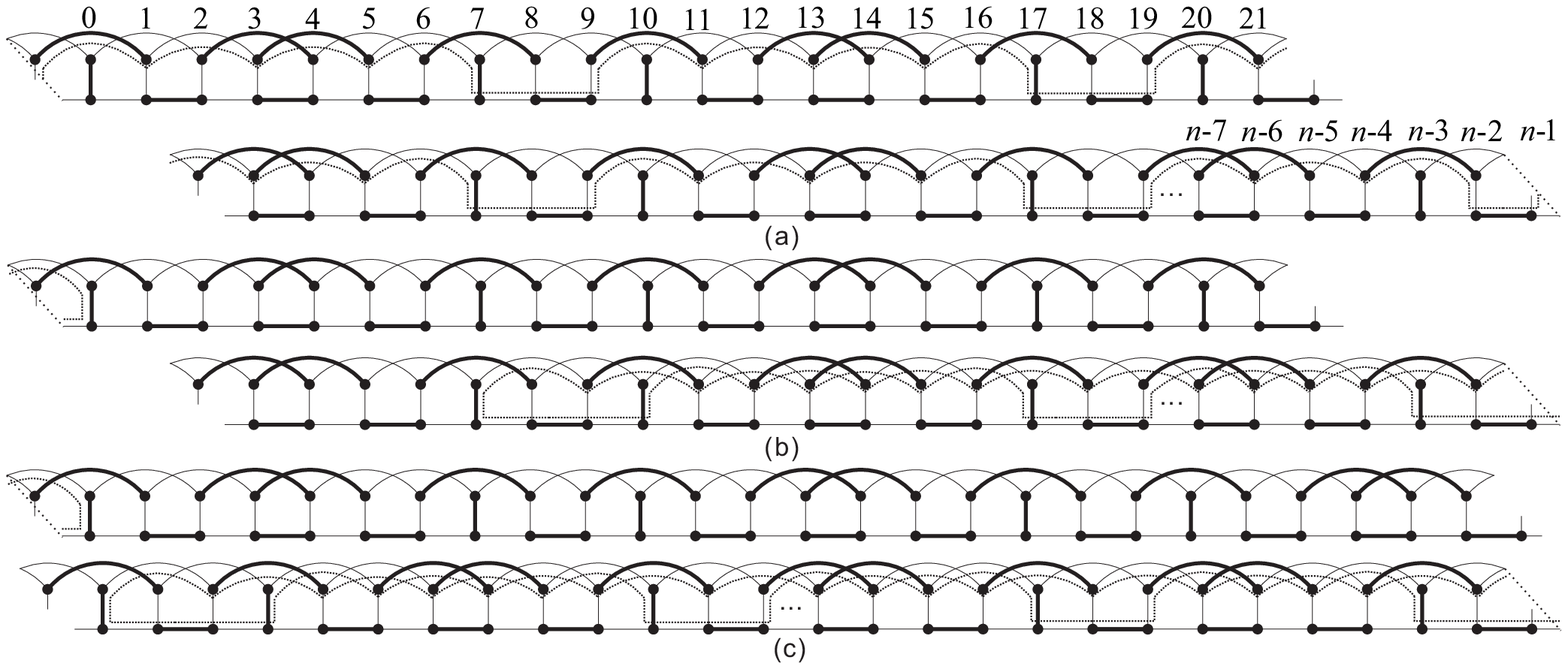}
            \caption{Illustration of Case 3 in the proof of Theorem \ref{thm2max}.}
            \label{upperkk11}
        \end{figure}

        \begin{cla2}
            \label{max33}
            Let $R$ be a subset of $V(P(n))$ with $u_{t-1},u_{t+2},u_{t+5}\in R$ and $u_{t+1},v_{t+1},v_{t+2},u_{t+3},$\\$v_{t+3}\notin R.$ Then $u_{t+1}u_{t+3}$ belongs to all perfect matchings of $P(n)-R$ if there exists one$.$
        \end{cla2}

        We illustrate the labels in Fig. \ref{algo5kk}. Because $u_{t+1}$, $v_{t+2}$, $u_{t+3}$ are odd components of $P(n)-R-u_{t+1}u_{t+3}-\{v_{t+1},v_{t+3}\}$, $P(n)-R-u_{t+1}u_{t+3}$ has no perfect matchings by Tutte's 1-factor Theorem. So the claim holds.

        \begin{figure}[htbp]
            \centering
            \includegraphics[height=0.45in]{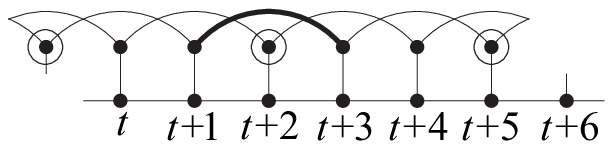}
            \caption{Illustration of Claim \ref{max33} in the proof of Theorem \ref{thm2max}.}
            \label{algo5kk}
        \end{figure}

        (2) Next we prove that for each $M\in \mathcal{M}_2(P(n))$ except $C(CD)^{\frac{n-3}7}$ ($n\equiv 3$ (mod 7)), there is a forcing set of $M$ with cardinality no more than $\lceil \frac { n+3 }{ 7 } \rceil$. Let us distinguish the number of $C$ in any such $M$ as follows.

        \textbf{Case 1.} There are no $C$ in $P(n)$. Then $n\equiv 0$ (mod 4) and $M=D^{\frac n 4}$. Let $S=\{u_{n-1}u_1,v_2v_3,u_{8i+4}u_{8i+6}:i=0,1,\ldots,\lfloor\frac{n-12}{8}\rfloor\}\subseteq M$ (see Fig. \ref{algo5kk11}(a)).

        \begin{cla2}
            \label{minkk}
            $S$ is a forcing set of $M$ of $P(n).$
        \end{cla2}

        We prove it by Corollary \ref{structure} (1) and (3) and Claim \ref{max33}. By a similar argument to Claim \ref{claim11} in Theorem \ref{thm1max}, we have $E_4=\{u_0u_2,v_0v_1,u_3u_5,v_4v_5\}$, $E_5=E_4\cup\{v_6v_7\}$, $E_6=E_5\cup\{u_7u_9\}$. By Claim \ref{max33}, $u_8u_{10}$ belongs to all perfect matchings of $P(n)-V(S)-V(E_6)$. Then $E_7=E_6\cup\{u_8u_{10}\}$. For the resulting graph, clearly we may continue to find new edges as stated in Corollary \ref{structure} until reaching an empty graph. Hence $S$ is a forcing set of $M$.

        \textbf{Case 2.} There is precisely one $C$ in $P(n)$. Similar as above, we can confirm that $\{u_{n-1}u_1,u_0v_0,u_{8i-1}u_{8i+1}:i=1,2,\ldots,\lfloor\frac {n-7} 8\rfloor\}$ (see Fig. \ref{algo5kk11}(b)) is a forcing set of $M$.

        \begin{figure}[htbp]
            \centering
            \includegraphics[height=1.8in]{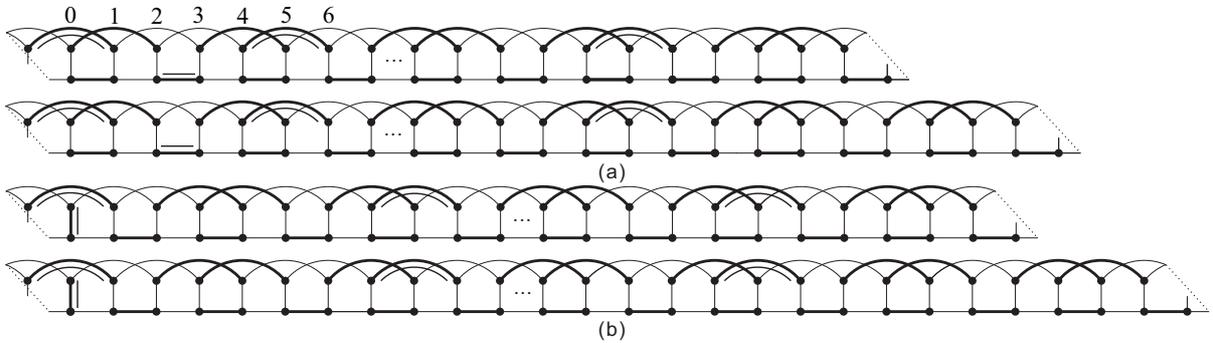}
            \caption{Illustration of Cases 1 and 2 in the proof of Theorem \ref{thm2max}.}
            \label{algo5kk11}
        \end{figure}

        \textbf{Case 3.} There are at least two $C$ in $P(n)$.

        \textbf{Case 3.1.} There is a chain $W$ in $P(n)$ being either $CCCC$, $CD^dC$ with even $d\ge 2$, or $CD^dC$ with odd $d\ge 5$.

        Suppose $P(n)-V(W)$ is empty. Let $S=\{u_{n-1}u_{1},u_0v_0,u_{n-3}v_{n-3},u_{8i-1}u_{8i+1}:i=1,2,\ldots,\lfloor\frac{d-1}{2}\rfloor\}$ (see Fig. \ref{algo5}(a)).

        \begin{cla2}
            \label{last}
            $S$ is a forcing set of $M$ of $P(n).$
        \end{cla2}

        We prove it by Corollary \ref{structure}. We know that $u_{n-4}u_{n-2}$ is a cut edge of $P(n)-V(S)$ such that $P(n)-V(S)-u_{n-4}u_{n-2}$ has an odd component with vertex set $\{v_{1},u_i,v_i:i=2,3,\ldots,n-4\}\setminus V(S)$. Then we can determine $E_1=\{u_{n-4}u_{n-2}\}$. For the resulting graph, similar to the proof of Claim \ref{minkk}, clearly we may continue to find new edges as stated in Corollary \ref{structure} until reaching an empty graph. Hence $S$ is a forcing set of $M$.

        From now on, suppose $P(n)-V(W)$ is not empty. Given a chain decomposition $W_1,W_2,\ldots,W_m$ ($m\ge 2$) such that $W_1$ is $W$; each $W_i$ ($2\le i\le m-1$) is either chain $DD,$ $DC,$ $CD,$ $C(CD)^kCC$ with $k\ge 0$, or $C(CD)^kD$ with $k\ge 1$; $W_m$ is either chain $D$, $DD,$ $DC,$ $CD,$ $C(CD)^k$ with $k\ge 0$, $C(CD)^kC$ with $k\ge 0$, $C(CD)^kCC$ with $k\ge 0$, or $C(CD)^kD$ with $k\ge 1$.

        For $W_1$, we give the edge subset $S_1$ in Fig. \ref{algo5}(b). For each $W_j$ ($2\le j\le m-1$), we give the edge subset $S_j$ in Fig. \ref{algo5}(c). For $W_m$, let $S_m=\emptyset$ if $W_m$ is chain $D$ or $DD$; otherwise, we give the edge subset $S_m$ in Fig. \ref{algo5}(d).

        Similar to the proof of Claim \ref{last}, we can confirm that $\cup_{j=1}^m S_j$ with cardinality no more than $\lceil \frac { n+3 }{ 7 } \rceil$ is a forcing set of $M$.

        \textbf{Case 3.2.} There is a chain $CDDDC$ in $P(n)$, denoted by $W$. Given a chain decomposition $W_1,W_2,\ldots,W_m$ ($m\ge 2$) similar to Case 3.1.

        If either there is a $W_i$ being chain $DD$ or $C(CD)^kCC$ with $k\ge 0$, or $W_m$ is not chain $DC$ or $CD$, then for $W_1$ let $S_1=\{u_{i-1}u_{i+1},u_iv_i,u_{i+7}u_{i+9},u_{i+15}v_{i+15}\}$ (see Fig. \ref{algo5}(b)); for others $W_j$, we give the edge subset $S_j$ similar to Case 3.1. Similar to the proof of Claim \ref{last}, we can confirm that $\cup_{j=1}^m S_j$ with cardinality no more than $\lceil \frac { n+3 }{ 7 } \rceil$ is a forcing set of $M$.

        Otherwise, each $W_i$ ($2\le i\le m-1$) is chain $DC$, $CD$ or $C(CD)^kD$ with $k\ge 1$, and $W_m$ is chain $DC$ or $CD$. Then for each $W_j$, we give the edge subset $S_j$ in Fig. \ref{algo5}(e). Similar to the proof of Claim \ref{last}, we can confirm that $\cup_{i=j}^m S_j$ with cardinality no more than $\lceil \frac { n+3 }{ 7 } \rceil$ is a forcing set of $M$.

        \textbf{Case 3.3.} There is a chain $CDC$ in $P(n)$, denoted by $W$. Given a chain decomposition $W_1,W_2,\ldots,W_m$ ($m\ge 2$) similar to Case 3.1.

        If either there is a $W_i$ being chain $DD$ or $C(CD)^kCC$ with $k\ge 0$, or $W_m$ is not chain $DC$ or $CD$, then for $W_1$ let $S_1=\{u_{i-1}u_{i+1},u_iv_i,u_{i+7}v_{i+7}\}$ (see Fig. \ref{algo5}(b)); for others $W_j$, we give the edge subset $S_j$ similar to Case 3.1. Similar to the proof of Claim \ref{last}, we can confirm that $\cup_{j=1}^m S_j$ with cardinality no more than $\lceil \frac { n+3 }{ 7 } \rceil$ is a forcing set of $M$.

        Otherwise, each $W_i$ ($2\le i\le m$) is chain $DC$ or $CD$, or rather, $M=C(CD)^{\frac{n-3}7}$.

        It remains to show that for the perfect matching $C(CD)^{\frac{n-3}7}$, there is a forcing set with cardinality no more than $\lceil \frac { n+3 }{ 7 } \rceil+1$. This can be verified by double lines illustrated in Fig. \ref{upper} for a similar reason to Claim \ref{last}.
    \end{proof}
    \begin{figure}
        \centering
        \includegraphics[height=6.7in]{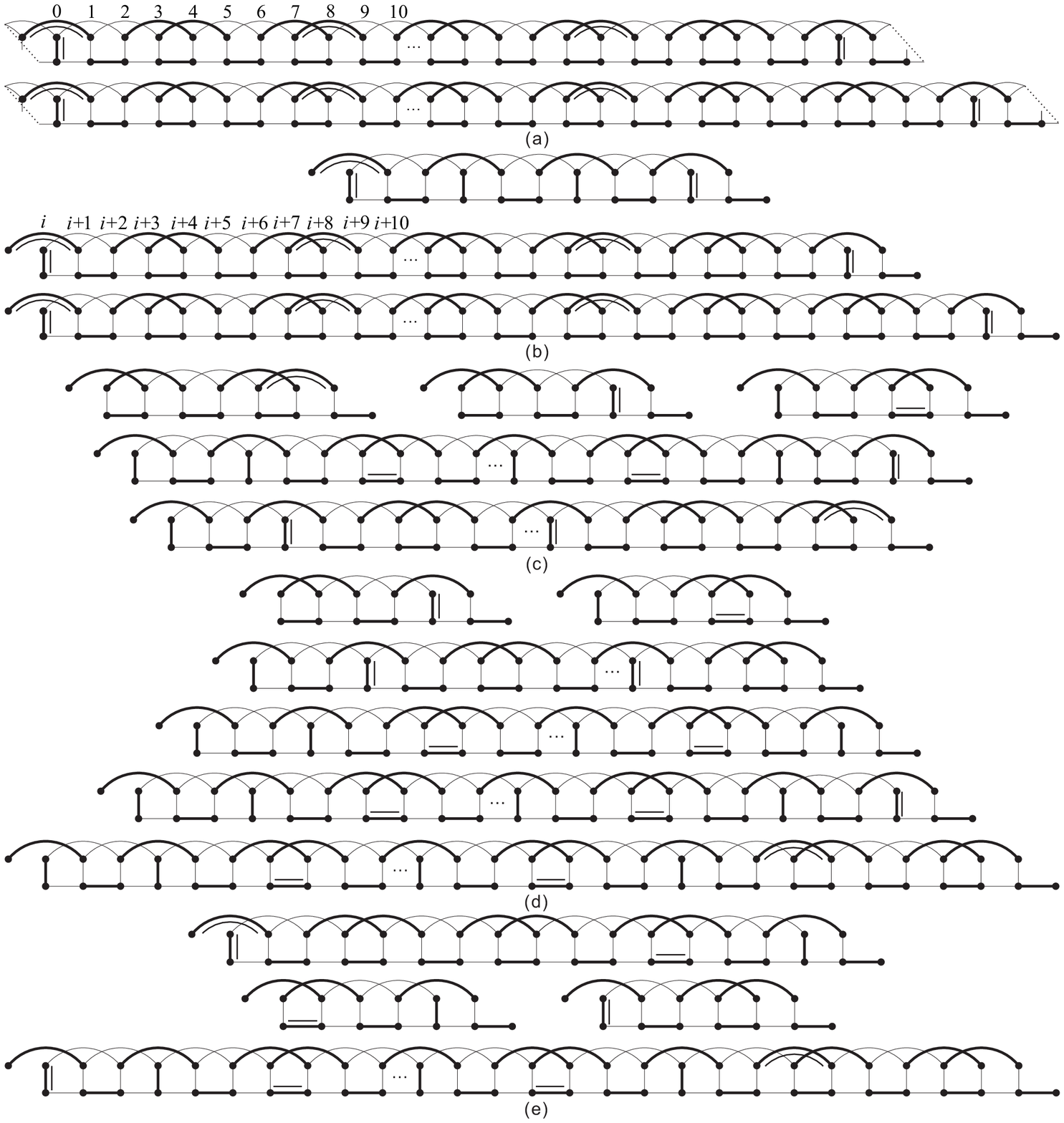}
        \caption{Illustration of Case 3 in the proof of Theorem \ref{thm2max}.}
        \label{algo5}
    \end{figure}

    Note that the above result does not hold for $n=33$.

\subsection{Minimum value of forcing numbers}
    \begin{thm}
        \label{thm2min}
        The minimum value of forcing numbers of perfect matchings in $\mathcal{M}_2(P(n))$ is $\lceil \frac { n }{ 12 } \rceil+1$ for $n\ge 11.$
    \end{thm}
    \begin{proof}
        We divide our proof in two steps. First we find a perfect matching in $\mathcal{M}_2(P(n))$ with forcing number no more than $\lceil \frac { n }{ 12 } \rceil+1$. Then we prove that the forcing number of each perfect matching in $\mathcal{M}_2(P(n))$ is no less than $\lceil \frac { n }{ 12 } \rceil+1$.

        (1) Now we give a perfect matching $M$ (see Fig. \ref{lower12}) of $P(n)$ expressed by
        \begin{align*}
        \left\{ \begin{array} {ll} CDD(CCCC)^{\lfloor\frac{n-11}{12}\rfloor} & \makebox{if}~n\equiv 11~(\text{mod}~12), \\  CCCC(CCCC)^{\lfloor\frac{n-11}{12}\rfloor} & \makebox{if}~n\equiv 0~(\text{mod}~12),\\ CCCD(CCCC)^{\lfloor\frac{n-11}{12}\rfloor} & \makebox{if}~n\equiv 1~(\text{mod}~12), \\  CCDD(CCCC)^{\lfloor\frac{n-11}{12}\rfloor} & \makebox{if}~n\equiv 2~(\text{mod}~12), \\ CCCCC(CCCC)^{\lfloor\frac{n-11}{12}\rfloor} & \makebox{if}~n\equiv 3~(\text{mod}~12), \\ CCCCD(CCCC)^{\lfloor\frac{n-11}{12}\rfloor} & \makebox{if}~n\equiv 4~(\text{mod}~12), \\ CCCDD(CCCC)^{\lfloor\frac{n-11}{12}\rfloor} & \makebox{if}~n\equiv 5~(\text{mod}~12), \\ CCCCCC(CCCC)^{\lfloor\frac{n-11}{12}\rfloor} & \makebox{if}~n\equiv 6~(\text{mod}~12), \\ CCCCCD(CCCC)^{\lfloor\frac{n-11}{12}\rfloor} & \makebox{if}~n\equiv 7~(\text{mod}~12), \\  CCCCDD(CCCC)^{\lfloor\frac{n-11}{12}\rfloor} & \makebox{if}~n\equiv 8~(\text{mod}~12), \\ CCCCCCC(CCCC)^{\lfloor\frac{n-11}{12}\rfloor} & \makebox{if}~n\equiv 9~(\text{mod}~12), \\  CDDCDD(CCCC)^{\lfloor\frac{n-11}{12}\rfloor} & \makebox{if}~n\equiv 10~(\text{mod}~12).          \end{array} \right.
        \end{align*}

        \begin{figure}[ht]
            \centering
            \includegraphics[height=4.95in]{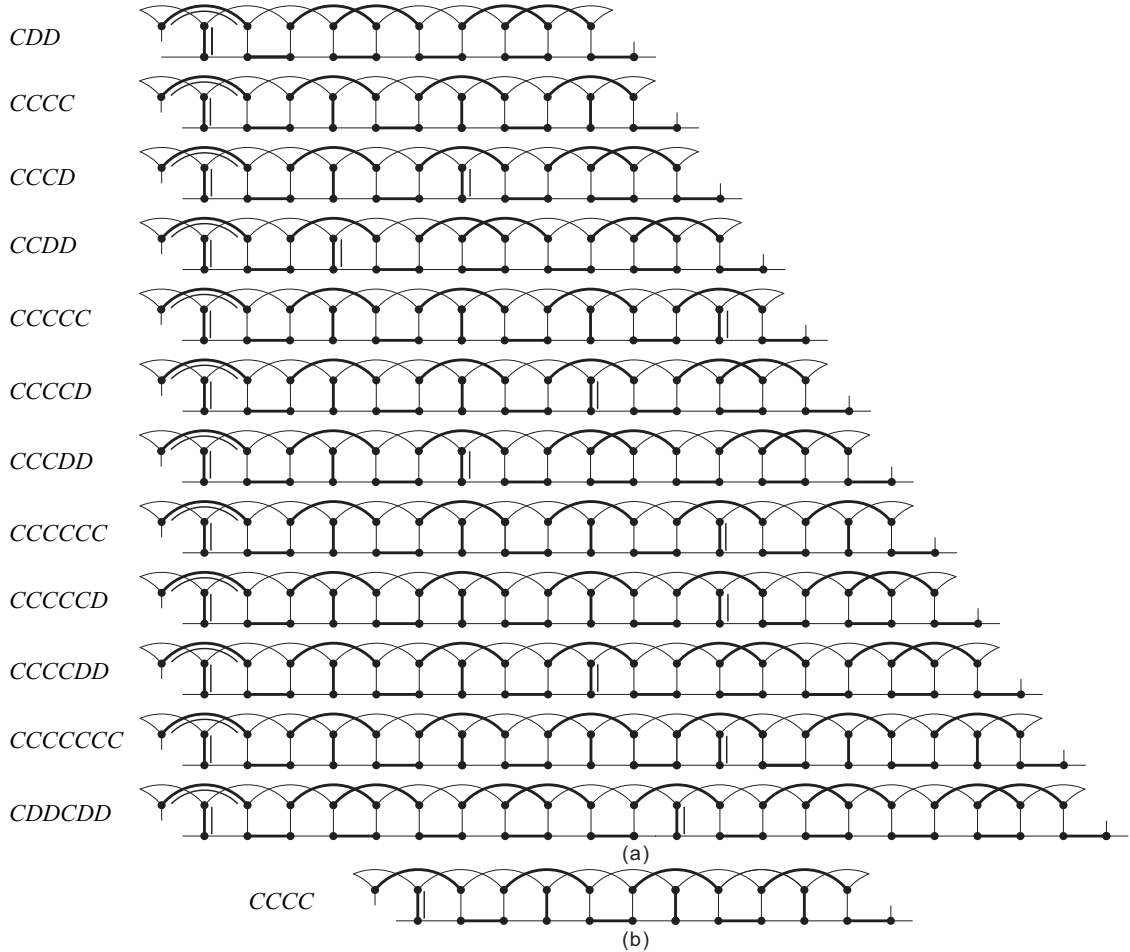}
            \caption{Perfect matching which achieves the lower bound of Theorem \ref{thm2min}.}
            \label{lower12}
        \end{figure}

        Similar to the proof of Claim \ref{last} in Theorem \ref{thm2max}, we can confirm that the edge subset in Fig. \ref{lower12} is a forcing set of $M$ of $P(n)$.

        (2) Next we prove that for each $M\in \mathcal{M}_2(P(n))$, we have $f(P(n),M)\ge\lceil \frac {n}{12} \rceil+1$. The initial cases of $11\le n\le 36$ can be verified from Table \ref{tab}. From now on suppose $n\ge 37$. To the contrary, suppose that $\mathcal{M}_2(P(n))$ has a perfect matching $M$ with a forcing set $S_0$ such that $|S_0|<\lceil \frac {n}{12} \rceil+1$. That is, $n-12|S_0|\ge -11$.

        Let us consider $P(n)$ with perfect matching $M$ as follows. We define an \emph{s-structure} to be a chain $C$ or $D$ which contains some edges of $S_0$, and an \emph{s-chain} to be a chain formed by an s-structure and its immediate left-hand maximal chain which contains no edges of $S_0$.

        Note that a chain contains no edges of $S_0$ if and only if it is either chain $C,$ $D,$ $CC,$ $DD$ or $CCC$ by Theorem \ref{lem5}. It follows that the number of edges of an s-chain $W$ in $M$ is no more than 13. Furthermore, it equals 13 if and only if $W$ is chain $CCCD$; it equals 12 if and only if $W$ is chain $CCCC$ or $DDD$; it equals 11 if and only if $W$ is chain $DDC$; it equals 10 if and only if $W$ is chain $CCD$; it equals 9 if and only if $W$ is chain $CCC$.

        Given an s-chain decomposition $W_1,W_2,\ldots,W_m$ ($m\ge 3$), let
        \[\beta (W_i)=|M\cap E(W_i)|-12|S_0\cap E(W_i)|\]
        for $i=1,2,\ldots,m$. Then
        \begin{align}
            \label{equu2}
            \beta (P(n)):=\sum_{i=1}^m\beta (W_i)=n-12|S_0|\ge -11.
        \end{align}

        Then we know each $W_i$ should be precisely one of the following cases.

        \textbf{Case 1.} There is a $W_i$ containing at least two edges of $S_0$. It follows that
        \begin{align*}
            &|M\cap E(P(n)-V(W_i))|= n-|M\cap E(W_i)|,\\
            &|S_0\cap E(P(n)-V(W_i))|\le |S_0|-2.
        \end{align*}
        Then by Eq. (\ref{equu2}), we have
        \begin{align*}
            \sum_{j\neq i}\beta(W_j)&=\beta (P(n)-V(W_i))\\
            &=|M\cap E(P(n)-V(W_i))|-12|S_0\cap E(P(n)-V(W_i))|\\
            &\ge n-|M\cap E(W_i)|-12|S_0|+24\\
            &\ge 13-|M\cap E(W_i)|\ge 0.
        \end{align*}
        Hence there are at least $13-|M\cap E(W_i)|$ s-chains $CCCD$ from $\{W_j:i\neq j=1,2,\ldots,m\}$ each containing precisely one edge of $S_0$.

        \textbf{Case 1.1.} $\sum_{j\neq i}\beta(W_j)\ge 1$. Without loss of generality, suppose $W_k$ as $P(n)[\{u_t,v_{t+1}:t=n-1,0,\ldots,11\}]$ is a chain $CCCD$ containing precisely one edge of $S_0$ ($k\neq i$). We claim that $|M\cap E(W_{k+1})|\le 8$ (the subscripts module $m$) and equality holds if and only if $W_{k+1}$ is chain $DD$. To the contrary, suppose $|M\cap E(W_{k+1})|> 8$. By Claim \ref{quan} in Theorem \ref{thm2max}, we can obtain an $M$-alternating cycle $C'$
        $u_{9}u_{11}u_{13}v_{13}v_{12}v_{11}v_{10}v_{9}u_{9}$ if $u_{8}u_{10}\in S_0$ and the first chain of $W_{k+1}$ is $C$, $u_{13}u_{15}u_{17}v_{17}v_{16}v_{15}v_{14}v_{13}u_{13}$ if $u_{8}u_{10}\in S_0$ and the first chain of $W_{k+1}$ is $DC$, $u_{9}u_{11}u_{13}u_{15}v_{15}v_{16}v_{17}v_{18}u_{18}u_{16}u_{14}u_{12}v_{12}v_{11}v_{10}v_{9}u_{9}$ if $u_{8}u_{10}\in S_0$ and the first chain of $W_{k+1}$ is $DD$, $u_{3}u_{5}u_{7}u_{9}u_{11}v_{11}v_{12}v_{13}v_{14}u_{14}u_{12}u_{10}u_{8}u_{6}v_{6}v_{5}v_{4}v_{3}u_{3}$ if $v_{9}v_{10}\in S_0$ and the first chain of $W_{k+1}$ is $D$, $u_{3}u_{5}u_{7}u_{9}u_{11}u_{13}v_{13}v_{14}v_{15}v_{16}u_{16}u_{14}u_{12}u_{10}u_{8}u_{6}v_{6}v_{5}v_{4}v_{3}u_{3}$ if $v_{9}v_{10}\in S_0$ and the first chain of $W_{k+1}$ is $CC$, and $u_{13}u_{15}u_{17}v_{17}v_{16}v_{15}v_{14}v_{13}u_{13}$ if $v_{9}v_{10}\in S_0$ and the first chain of $W_{k+1}$ is $CD$. In all cases mentioned above, the obtained $M$-alternating cycle $C'$ contains no edges of $S_0$, a contradiction.

        Note that if $W_{k+1}=W_i$, then we could find another s-chain $CCCD$ different from $W_k$ and $W_i$ containing precisely one edge of $S_0$, which satisfies the above argument as $W_k$. So we may assume that $W_{k+1}\neq W_i$. Hence
        \begin{align*}
            \sum_{j\neq i,k,k+1}\beta(W_j)&=\beta (P(n)-V(W_i)-V(W_k)-V(W_{k+1}))\\
            =&|M\cap E(P(n)-V(W_i)-V(W_k)-V(W_{k+1}))|\\
            &-12|S_0\cap E(P(n)-V(W_i)-V(W_k)-V(W_{k+1}))|\\
            =&(|M\cap E(P(n)-V(W_i))|-|M\cap E(W_k)|-|M\cap E(W_{k+1})|)\\
            &-12(|S_0\cap E(P(n)-V(W_i))|-|S_0\cap E(W_k)|-|S_0\cap E(W_{k+1})|)\\
            \ge& (|M\cap E(P(n)-V(W_i))|-13-8)-12(|S_0\cap E(P(n)-V(W_i))|-1-1)\\
            \ge& 16-|M\cap E(W_i)|\ge 3.
        \end{align*}
        It follows that there are at least $16-|M\cap E(W_i)|$ s-chains $CCCD$ from $\{W_j:j=1,2,\ldots,m,j\neq i,k,k+1\}$ each containing precisely one edge of $S_0$. Iterating the above procedure, this deduces a contradiction to $n$ being finite by a similar argument as above.

        \textbf{Case 1.2.} $\sum_{j\neq i}\beta(W_j)=0$. Similarly, we know that $W_i$ is the only chain $CCCD$ which contains precisely two edges of $S_0$, and others $W_j$ must be chain $CCCC$ or $DDD$ which contain precisely one edge of $S_0$, respectively. Without loss of generality, let $W_i=P(n)[\{u_t,v_{t+1}:t=n-1,0,\ldots,11\}]$. Since there is an $M$-alternating path $u_{n-1}u_{1}u_3v_3v_4v_5u_5u_7v_7v_8u_8u_{10}u_{12}$ containing no edge $u_9u_{11},$ $v_9v_{10}$ or $v_{11}v_{12}$, an $M$-alternati-\\ng path $u_{n-1}u_{1}u_3v_3v_4v_5u_5u_7u_9u_{11}u_{13}$ containing no edge $u_8u_{10},$ $v_9v_{10}$ or $v_{11}v_{12}$, an $M$-alternating path $u_{n-1}u_{1}u_3v_3v_4v_5u_5u_7v_7v_8v_9v_{10}v_{11}v_{12}$ containing no edge $u_8u_{10}$ or $u_9u_{11}$, an $M$-alternating path $u_{12j}u_{12j+2}u_{12j+4}v_{12j+4}v_{12j+5}v_{12j+6}v_{12j+7}u_{12j+7}u_{12j+9}u_{12j+11}v_{12j+11}$\\$v_{12j+12}$ containing no edge $u_{12j+10}v_{12j+10}$ if $W_{i+j}$ (the subscripts modulo $m$) is chain $CCCC$, an $M$-alternating path $u_{12j+1}v_{12j+1}v_{12j+2}v_{12j+3}u_{12j+3}u_{12j+5}v_{12j+5}v_{12j+6}v_{12j+7}$\\$u_{12j+7}u_{12j+9}u_{12j+11}v_{12j+11}v_{12j+12}$ containing no edge $u_{12j+10}v_{12j+10}$ if $W_{i+j}$ is chain $CCCC$, an $M$-alternating path $u_{12j}u_{12j+2}u_{12j+4}u_{12j+6}v_{12j+6}v_{12j+5}u_{12j+5}u_{12j+7}u_{12j+9}u_{12j+11}v_{12j+11}$\\$v_{12j+12}$ containing no edge $u_{12j+8}u_{12j+10}$ or $v_{12j+9}v_{12j+10}$ if $W_{i+j}$ is chain $DDD$, and an $M$-alternating path $u_{12j+1}u_{12j+3}v_{12j+3}v_{12j+4}u_{12j+4}u_{12j+6}v_{12j+6}v_{12j+5}u_{12j+5}u_{12j+7}u_{12j+9}u_{12j+11}$\\$v_{12j+11}v_{12j+12}$ containing no edge $u_{12j+8}u_{12j+10}$ or $v_{12j+9}v_{12j+10}$ if $W_{i+j}$ is chain $DDD$, we can generate an $M$-alternating cycle which contains no edges of $S_0$, a contradiction.

        \textbf{Case 2.} Each $W_i$ contains precisely one edge of $S_0$. By a similar argument to Case 1.2, we have that there is a $W_j$ being chain $C$ or $D$. Then
        \begin{align*}
            &|M\cap E(P(n)-V(W_j))|\ge n-4,\\
            &|S_0\cap E(P(n)-V(W_j))|=|S_0|-1.
        \end{align*}
        Then by Eq. (\ref{equu2}), we have
        \begin{align*}
            \sum_{k\neq j}\beta(W_k)=\beta (P(n)-V(W_j))\ge n-4-12|S_0|+12\ge -3.
        \end{align*}
        By a similar argument to Case 1.1, we have $\sum_{k\neq j}\beta(W_k)\le 0$, and there are no s-chains $CCCD$ except for the case of $\sum_{k\neq j}\beta(W_k)=-3$ with precisely one s-chain $CCCD$ from $\{W_i:j\neq i=1,2,\ldots m\}$. Furthermore, if $\sum_{k\neq j}\beta(W_k)=0$, then $W_j$ is chain $C$ or $D$ and others $W_k$ are all chains $CCCC$ or $DDD$; if $\sum_{k\neq j}\beta(W_k)=-1$, then $W_j$ is chain $C$ or $D$ and others $W_k$ are chains $CCCC$ or $DDD$ except for one being $DDC$; if $\sum_{k\neq j}\beta(W_k)=-2$, then $W_j$ is chain $C$ or $D$ and others $W_k$ are one of the two cases: (1) chains $CCCC$ or $DDD$ except for one being $CCD$, and (2) chains $CCCC$ or $DDD$ except for two being $DDC$; if $\sum_{k\neq j}\beta(W_k)=-3$, then $W_j$ is chain $D$ and others $W_k$ are one of the four cases: (1) chains $CCCC$ or $DDD$ except for continuous two being chains $CCCD$ and $DD$, (2) chains $CCCC$ or $DDD$ except for one being $CCC$, (3) chains $CCCC$ or $DDD$ except for two with one being $CCD$ and the other being $DDC$, and (4) chains $CCCC$ or $DDD$ except for three being $DDC$. Similar to Case 1.2, all cases mentioned above deduce a contradiction.
    \end{proof}

    Note that the above result does not hold for $n=10$.

\subsection{Continuity}
    \begin{thm}
        \label{thm2con}
        For $n\ge 3,$ $\{f(P(n),M): M \in \mathcal{M}_2(P(n))\}$ is continuous$.$
    \end{thm}

    In order to prove the theorem, we need the following two lemmas.

    \begin{lem}
        \label{conclusion1}
        For $n\ge 37,$ let $M_1$ be a perfect matching expressed by a sequence of at least one $C$ and at least one $D,$ $M_2$ be the perfect matching obtained from $M_1$ by transforming one chain $CD$ to $DC,$ and maintaining the other parts$.$ Then
        \[|f(P(n),M_1)-f(P(n),M_2)|\le 1.\]
    \end{lem}
    \begin{proof}
        We illustrate the labels in Fig. \ref{upper11}. In fact, $M_2$ is the symmetric difference between $M_1$ and the $M_1$-alternating cycle $u_{i+1}u_{i+3}u_{i+5}v_{i+5}v_{i+4}v_{i+3}v_{i+2}v_{i+1}u_{i+1}$. Denote the subgraph $P(n)[\{u_s,v_{s+1}:s=i,i+1,\ldots,i+6\}]$ by $W$, and the immediate left-hand and right-hand chains $C$ or $D$ of $W$ by $U$ and $V$, respectively.

        For each minimum forcing set $S_1$ of $M_1$, the number of edges of $W$ in $S_1$ is no less than one by Theorem \ref{lem5}, and no more than three since the edges $u_{i}u_{i+2},$ $u_{i+1}v_{i+1},$ $u_{i+4}u_{i+6}$ can determine all edges of $M_1\cap E(W)$.

        If $|S_1\cap E(W)|\ge 2$, then we could obtain a forcing set of $M_2$ from $S_1$ by transforming all edges in $S_1\cap E(W)$ to three edges $u_{i}u_{i+2},$ $u_{i+4}u_{i+6},$ $u_{i+5}v_{i+5}$ and maintaining the other edges, which implies $f(P(n),M_2)\le f(P(n),M_1)+1$.

        If $|S_1\cap E(W)|=1$, then $S_1\cap E(W)\subset \{u_{i+1}v_{i+1},v_{i+2}v_{i+3},v_{i+4}v_{i+5},u_{i+3}u_{i+5}\}$ by Claim \ref{quan} in Theorem \ref{thm2max}. If $S_1\setminus E(W)$ is contained in no perfect matchings in $\mathcal{M}_1(P(n))$, then we could obtain a forcing set of $M_2$ from $S_1$ by transforming the edge in $S_1\cap E(W)$ to two edges $u_{i}u_{i+2},$ $u_{i+5}v_{i+5}$ and maintaining the other edges, which implies $f(P(n),M_2)\le f(P(n),M_1)+1$. We now consider the case that $S_1\setminus E(W)$ is contained in some perfect matching in $\mathcal{M}_1(P(n))$ according to $|S_1\cap (E(U)\cup E(V))|:=k$ as follows.

        If $k\ge 2$, then we could obtain a forcing set of $M_2$ from $S_1$ by transforming the $k$ edges and the edge in $S_1\cap E(W)$ to four edges $u_{i+5}v_{i+5}$, $u_{i+4}u_{i+6}$, $u_{i-2}v_{i-2}$, $u_{i+8}v_{i+8}$ and maintaining the other edges if $U$ is $C$ and $V$ is $C$, four edges $u_{i+5}v_{i+5}$, $u_{i+4}u_{i+6}$, $u_{i-2}v_{i-2}$, $u_{i+8}u_{i+10}$ and maintaining the other edges if $U$ is $C$ and $V$ is $D$, four edges $u_{i+5}v_{i+5}$, $u_{i+4}u_{i+6}$, $u_{i-4}u_{i-2}$, $u_{i+8}v_{i+8}$ and maintaining the other edges if $U$ is $D$ and $V$ is $C$, and four edges $u_{i+5}v_{i+5}$, $u_{i+4}u_{i+6}$, $u_{i-4}u_{i-2}$, $u_{i+8}u_{i+10}$ and maintaining the other edges if $U$ is $D$ and $V$ is $D$. If $k=1$, then $|S_1\cap E(U)|=1$, since if otherwise, then $S_1$ is contained in some perfect matching in $\mathcal{M}_1(P(n))$, a contradiction to $S_1$ being a forcing set of $M_1$. We could obtain a forcing set of $M_2$ from $S_1$ by transforming two edges in $S_1\cap (E(U)\cup E(W))$ to three edges $u_{i+5}v_{i+5}$, $u_{i+4}u_{i+6}$, $u_{i-2}v_{i-2}$ and maintaining the other edges if $U$ is $C$, and three edges $u_{i+5}v_{i+5}$, $u_{i+4}u_{i+6}$, $u_{i-4}u_{i-2}$ and maintaining the other edges if $U$ is $D$. If $k=0$, then $S_1$ is contained in some perfect matching in $\mathcal{M}_1(P(n))$, a contradiction. These imply $f(P(n),M_2)\le f(P(n),M_1)+1$.

        Similarly, we could obtain $f(P(n),M_1)\le f(P(n),M_2)+1$.
    \end{proof}

    \begin{lem}
        \label{conclusion}
        For $d\ge 0,$ $c\ge 4$ and $n\ge 37,$ let $M_3=D^{d}C^{c},$ $M_4=D^{d+3}C^{c-4}.$ Then
        \[|f(P(n),M_3)-f(P(n),M_4)|\le 1.\]
    \end{lem}

    \begin{proof}
        For convenience, we assume that $M_4$ is obtained from $M_3$ by transforming one chain $C^4$ to $D^3$. We illustrate the labels in Fig. \ref{con2}. In fact, $M_4$ is the symmetric difference between $M_3$ and the $M_3$-alternating cycle $u_{j+1}u_{j+3}u_{j+5}u_{j+7}v_{j+7}v_{j+8}v_{j+9}v_{j+10}u_{j+10}u_{j+8}u_{j+6}$\\$u_{j+4}v_{j+4}v_{j+3}v_{j+2}v_{j+1}u_{j+1}$. Denote the subgraph $P(n)[\{u_s,v_{s+1}:s=j,j+1,\ldots,j+11\}]$ by $W'$, the immediate left-hand and right-hand chains $C$ or $D$ of $W'$ by $U'$ and $V'$ respectively, and the immediate left-hand (resp. right-hand) chain $C$ or $D$ of $U'$ (resp. $V'$) by $U''$ (resp. $V''$). Note that if $U'$ is $C$, then $V'$ is also $C$. Furthermore, if $U'$ and $V'$ both are $D$, then $M_4=D^{\frac n 4}$; if $U'$ and $V'$ both are $C$, then $M_3=C^{\frac n 3}$.

        \begin{figure}[ht]
            \centering
            \includegraphics[height=0.45in]{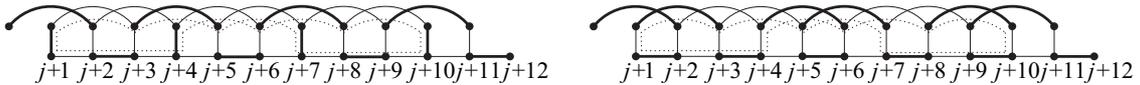}
            \caption{Alternating cycles in chains $C^4$ and $D^3$.}
            \label{con2}
        \end{figure}

        For each minimum forcing set $S_3$ of $M_3$, the number of edges of $W'$ in $S_3$ is no less than one by Theorem \ref{lem5}, and no more than three since the edges $u_{j}u_{j+2},$ $u_{j+1}v_{j+1},$ $u_{j+10}v_{j+10}$ can determine all edges of $M_3\cap E(W')$. Similarly, for each minimum forcing set $S_4$ of $M_4$, the number of edges of $W'$ in $S_4$ is no less than one, and no more than four since the edges $u_{j}u_{j+2},$ $u_{j+1}u_{j+3},$ $v_{j+1}v_{j+2}$, $u_{j+9}u_{j+11}$ can determine all edges of $M_4\cap E(W')$.

        First we prove that $f(P(n),M_3)\le f(P(n),M_4)+1$.

        If $|S_4\cap E(W')|\ge 2$, then we could obtain a forcing set of $M_3$ from $S_4$ by transforming all edges in $S_4\cap E(W')$ to three edges $u_{j}u_{j+2},$ $u_{j+1}v_{j+1},$ $u_{j+10}v_{j+10}$ and maintaining the other edges, which implies $f(P(n),M_3)\le f(P(n),M_4)+1$.

        Suppose $|S_4\cap E(W')|=1$. If $V'$ is $D$, then we could transform this case to the case of $|S_4\cap E(W')|\ge 2$ by selecting appropriate three continuous chains $D$ to replace $W'$ (for otherwise, $S_4$ is contained in some perfect matching in $\mathcal{M}_1(P(n))$). Denote the edge of $W'$ in $S_4$ by $e$. Then $e\in\{u_{j+1}u_{j+3},v_{j+1}v_{j+2},v_{j+3}v_{j+4},u_{j+4}u_{j+6},u_{j+5}u_{j+7},v_{j+7}v_{j+8},u_{j+8}u_{j+10},$\\$v_{j+9}v_{j+10}\}$ by Claim \ref{quan} in Theorem \ref{thm2max}. If $S_4\setminus \{e\}$ is contained in no perfect matchings in $\mathcal{M}_1(P(n))$, then we could obtain a forcing set of $M_3$ from $S_4$ by transforming $e$ to two edges $u_{j+1}v_{j+1},$ $u_{j+10}v_{j+10}$ and maintaining the other edges, which implies $f(P(n),M_3)\le f(P(n),M_4)+1$. We now consider the case that $V'$ is $C$ and $S_4\setminus \{e\}$ is contained in some perfect matching in $\mathcal{M}_1(P(n))$ according to $|S_4\cap (E(U')\cup E(V'))|:=k_4$ as follows.

        If $k_4\ge 2$, then we could obtain a forcing set of $M_3$ from $S_4$ by transforming the $k_4$ edges and $e$ to four edges $v_{j-3}v_{j-2}$, $u_{j}u_{j+2}$, $u_{j+1}v_{j+1}$, $u_{j+13}v_{j+13}$ and maintaining the other edges if $U'$ is $D$, and four edges $u_{j-2}v_{j-2}$, $u_{j}u_{j+2}$, $u_{j+1}v_{j+1}$, $u_{j+13}v_{j+13}$ and maintaining the other edges if $U'$ is $C$, which implies $f(P(n),M_3)\le f(P(n),M_4)+1$.

        Suppose $k_4=1$. If $|S_4\cap E(U')|=1$, then $e=v_{j+9}v_{j+10}$ by Claim \ref{quan} in Theorem \ref{thm2max}, which implies that $S_4$ is contained in some perfect matching in $\mathcal{M}_1(P(n))$, a contradiction. If $|S_4\cap E(V')|=1$, then $e\in \{u_{j+1}u_{j+3},v_{j+1}v_{j+2},v_{j+3}v_{j+4},u_{j+4}u_{j+6}\}$ and $S_4\cap E(V')=\{u_{j+13}v_{j+13}\}$, which implies that $S_4$ is contained in some perfect matching in $\mathcal{M}_1(P(n))$, a contradiction.

        If $k_4=0$, then $S_4$ is contained in some perfect matching in $\mathcal{M}_1(P(n))$, a contradiction.

        Next we prove that $f(P(n),M_4)\le f(P(n),M_3)+1$.

        \textbf{Case 1.} $|S_3\cap E(W')|=3$. Then we could obtain a forcing set of $M_4$ from $S_3$ by transforming all edges in $S_3\cap E(W')$ to four edges $u_{j}u_{j+2},$ $u_{j+1}u_{j+3},$ $v_{j+1}v_{j+2}$, $u_{j+9}u_{j+11}$ and maintaining the other edges, which implies $f(P(n),M_4)\le f(P(n),M_3)+1$.

        \textbf{Case 2.} $|S_3\cap E(W')|=2$. If $S_3\setminus E(W')$ is contained in no perfect matchings in $\mathcal{M}_1(P(n))$, then we could obtain a forcing set of $M_4$ from $S_3$ by transforming all edges in $S_3\cap E(W')$ to three edges $u_{j}u_{j+2},$ $v_{j+1}v_{j+2}$, $u_{j+9}u_{j+11}$ and maintaining the other edges, which implies $f(P(n),M_4)\le f(P(n),M_3)+1$. We now consider the case that $S_3\setminus E(W')$ is contained in some perfect matching in $\mathcal{M}_1(P(n))$ according to $|S_3\cap (E(U')\cup E(V'))|:=k_3$ as follows.

        \textbf{Case 2.1.} $k_3\ge 2$. Then we could obtain a forcing set of $M_4$ from $S_3$ by transforming all edges in $S_3\cap (E(U')\cup E(W')\cup E(V'))$ to four edges $u_{j-4}u_{j-2},$ $v_{j-1}v_{j}$, $u_{j+1}u_{j+3}$, $u_{j+13}v_{j+13}$ and maintaining the other edges if $U'$ is $D$ and $V'$ is $C$, five edges $u_{j-4}u_{j-2},$ $v_{j-1}v_{j}$, $u_{j+1}u_{j+3}$, $u_{j+9}u_{j+11}$, $v_{j+13}v_{j+14}$ and maintaining the other edges if $U'$ is $D$ and $V'$ is $D$, and four edges $u_{j-3}u_{j-1},$ $u_{j-2}v_{j-2}$, $v_{j+9}v_{j+10}$, $u_{j+13}v_{j+13}$ and maintaining the other edges if $U'$ is $C$ and $V'$ is $C$, which implies $f(P(n),M_4)\le f(P(n),M_3)+1$.

        \textbf{Case 2.2.} $k_3=1$. If $V'$ is $C$, then we could obtain a forcing set of $M_4$ from $S_3$ similar as above; if $V'$ is $D$ and $|S_3\cap E(U')|=1$, then we could obtain a forcing set of $M_4$ from $S_3$ by transforming all edges in $S_3\cap (E(U')\cup E(W'))$ to four edges $u_{j-4}u_{j-2},$ $v_{j-1}v_{j}$, $u_{j+1}u_{j+3}$, $u_{j+9}u_{j+11}$ and maintaining the other edges; if $V'$ is $D$ and $|S_3\cap E(V')|=1$, then we could obtain a forcing set of $M_4$ from $S_3$ by transforming all edges in $S_3\cap (E(W')\cup E(V'))$ to four edges $u_{j}u_{j+2},$ $v_{j+3}v_{j+4}$, $u_{j+5}u_{j+7}$, $u_{j+13}u_{j+15}$ and maintaining the other edges, which implies $f(P(n),M_4)\le f(P(n),M_3)+1$.

        \textbf{Case 2.3.} $k_3=0$. If $U'$ is $D$, then the two edges in $S_3\cap E(W')$ are contained in different chains $C$, which implies that $S_3$ is contained in some perfect matching in $\mathcal{M}_1(P(n))$, a contradiction. We now assume that $U'$ is $C$.

        If $|S_3\cap (E(U'')\cup E(V''))|\ge 2$, then we could transform this case to the case of $|S_3\cap E(W')|= 3$ by changing all edges in $S_3\cap (E(U'')\cup E(W')\cup E(V''))$ into four edges $u_{j-5}v_{j-5}$, $u_{j+7}v_{j+7}$, $u_{j+6}u_{j+8}$, $u_{j+16}v_{j+16}$ and maintaining the other edges; if $|S_3\cap (E(U'')\cup E(V''))|=0$, then $S_3$ is contained in some perfect matching in $\mathcal{M}_1(P(n))$, a contradiction; if $|S_3\cap (E(U'')\cup E(V''))|=1$ (w.l.o.g. suppose $|S_3\cap E(U'')|=1$), then we have $|S_3\cap E(V''')|\ge 1$, where $V'''$ is the immediate right-hand chain $C$ of $V''$ (for otherwise, $S_3$ is contained in another perfect matching in $\mathcal{M}_2(P(n))$). We could transform this case to the case of $|S_3\cap E(W')|= 2$ and $k_3=1$ by changing all edges in $S_3\cap (E(U'')\cup E(W')\cup E(V'''))$ into four edges $u_{j-5}v_{j-5}$, $u_{j+7}v_{j+7}$, $u_{j+6}u_{j+8}$, $u_{j+19}v_{j+19}$ and maintaining the other edges. These imply $f(P(n),M_4)\le f(P(n),M_3)+1$.

        \textbf{Case 3.} $|S_3\cap E(W')|=1$. If $U'$ is $C$, then we could transform this case to the case of $|S_3\cap E(W')|\ge 2$ by selecting appropriate four continuous chains $C$ to replace $W'$ (for otherwise, $S_3$ is contained in some perfect matching in $\mathcal{M}_1(P(n))$). We now consider the case that $U'$ is $D$ according to $k_3$ as follows.

        \textbf{Case 3.1.} $k_3\ge 3$. Then we could obtain a forcing set of $M_4$ from $S_3$ by transforming all edges in $S_3\cap (E(U')\cup E(W')\cup E(V'))$ to four edges $u_{j-4}u_{j-2},$ $v_{j-1}v_{j}$, $u_{j+1}u_{j+3}$, $u_{j+13}v_{j+13}$ and maintaining the other edges if $V'$ is $C$, and five edges $u_{j-4}u_{j-2},$ $v_{j-1}v_{j}$, $u_{j+1}u_{j+3}$, $u_{j+9}u_{j+11}$, $v_{j+13}v_{j+14}$ and maintaining the other edges if $V'$ is $D$, which implies $f(P(n),M_4)\le f(P(n),M_3)+1$.

        \textbf{Case 3.2.} $k_3=2$. If $V'$ is $C$, then we could obtain a forcing set of $M_4$ from $S_3$ similar as above. Suppose $V'$ is $D$. If $|S_3\cap E(U')|=2$, then we could obtain a forcing set of $M_4$ from $S_3$ by transforming all edges in $S_3\cap (E(U')\cup E(W'))$ to four edges $u_{j-4}u_{j-2},$ $v_{j-1}v_{j}$, $u_{j+1}u_{j+3}$, $u_{j+9}u_{j+11}$ and maintaining the other edges; if $|S_3\cap E(V')|=2$, then we could obtain a forcing set of $M_4$ from $S_3$ by transforming all edges in $S_3\cap (E(W')\cup E(V'))$ to four edges $u_{j}u_{j+2},$ $v_{j+3}v_{j+4}$, $u_{j+5}u_{j+7}$, $u_{j+13}u_{j+15}$ and maintaining the other edges. We now assume that $|S_3\cap E(U')|=|S_3\cap E(V')|=1$.

        If $S_3\setminus (E(U')\cup E(W')\cup E(V'))$ is contained in no perfect matchings in $\mathcal{M}_1(P(n))$, then we could obtain a forcing set of $M_4$ from $S_3$ by transforming all edges in $S_3\cap (E(U')\cup E(W')\cup E(V'))$ to four edges $u_{j-4}u_{j-2}$, $v_{j-3}v_{j-2}$, $u_{j+5}u_{j+7}$, $u_{j+13}u_{j+15}$ and maintaining the other edges; otherwise, then $|S_3\cap (E(U'')\cup E(V''))|\ge 1$ (for otherwise, $S_3$ is contained in some perfect matching in $\mathcal{M}_1(P(n))$). Without loss of generality, suppose $|S_3\cap E(U'')|\ge 1$. Then we could obtain a forcing set of $M_4$ from $S_3$ by transforming all edges in $S_3\cap (E(U'')\cup E(U')\cup E(W')\cup E(V'))$ to five edges $u_{j-8}u_{j-6},$ $v_{j-5}v_{j-4}$, $u_{j-3}u_{j-1}$, $u_{j+5}u_{j+7}$, $u_{j+13}u_{j+15}$ and maintaining the other edges. These imply $f(P(n),M_4)\le f(P(n),M_3)+1$.

        \textbf{Case 3.3.} $k_3=1$. Suppose $V'$ is $C$. If $S_3\setminus (E(U')\cup E(W')\cup E(V'))$ is contained in no perfect matchings in $\mathcal{M}_1 (P(n))$, then we could obtain a forcing set of $M_4$ from $S_3$ by transforming all edges in $S_3\cap (E(U')\cup E(W')\cup E(V'))$ to three edges $u_{j-4}u_{j-2}$, $u_{j+4}u_{j+6}$, $u_{j+13}v_{j+13}$ and maintaining the other edges; otherwise, we have $|S_3\cap E(U')|=1$ and $|S_3\cap E(U'')|\ge 1$ (for otherwise, $S_3$ is contained in some perfect matching in $\mathcal{M}_1(P(n))$). If $U''$ is $C$, then we could obtain a forcing set of $M_4$ from $S_3$ by transforming all edges in $S_3\cap (E(U'')\cup E(U')\cup E(W'))$ to four edges $u_{j-6}v_{j-6}$, $u_{j+4}u_{j+6}$, $v_{j+7}v_{j+8}$, $u_{j+9}u_{j+11}$ and maintaining the other edges. We now assume that $U''$ is $D$.

        If $|S_3\cap E(U'')|\ge 2$, then we could obtain a forcing set of $M_4$ from $S_3$ by transforming all edges in $S_3\cap (E(U'')\cup E(U')\cup E(W'))$ to five edges $u_{j-8}u_{j-6}$, $u_{j-4}u_{j-2}$, $u_{j+4}u_{j+6}$, $v_{j+7}v_{j+8}$, $u_{j+9}u_{j+11}$ and maintaining the other edges; if $|S_3\cap E(U'')|=1$, then $|S_3\cap E(U''')|\ge 1$, where $U'''$ is the immediate left-hand chain $C$ or $D$ of $U''$ (for otherwise, $S_3$ is contained in some perfect matching in $\mathcal{M}_1(P(n))$). We could obtain a forcing set of $M_4$ from $S_3$ by transforming all edges in $S_3\cap (E(U''')\cup E(U'')\cup E(U')\cup E(W'))$ to five edges $u_{j-12}u_{j-10}$, $u_{j-4}u_{j-2}$, $u_{j+4}u_{j+6}$, $v_{j+7}v_{j+8}$, $u_{j+9}u_{j+11}$ and maintaining the other edges if $U'''$ is $D$, and five edges $u_{j-10}v_{j-10}$, $u_{j-4}u_{j-2}$, $u_{j+4}u_{j+6}$, $v_{j+7}v_{j+8}$, $u_{j+9}u_{j+11}$ and maintaining the other edges if $U'''$ is $C$. These imply $f(P(n),M_4)\le f(P(n),M_3)+1$.

        Suppose $V'$ is $D$. If $S_3\setminus (E(U')\cup E(W')\cup E(V'))$ is contained in some perfect matchings in $\mathcal{M}_1 (P(n))$, then $|S_3\cap E(U')|= 1$ and $|S_3\cap E(U'')|\ge 1$ (for otherwise, $S_3$ is contained in some perfect matching in $\mathcal{M}_1(P(n))$). By a similar argument as above, we have $f(P(n),M_4)\le f(P(n),M_3)+1$. We now suppose that $S_3\setminus (E(U')\cup E(W')\cup E(V'))$ is contained in no perfect matchings in $\mathcal{M}_1 (P(n))$.

        If $|S_3\cap E(U')|=1$, then we have $S_3\cap E(U')\subset \{u_{j-3}u_{j-1},v_{j-3}v_{j-2},v_{j-1}v_j\}$ and $S_3\cap E(W')=\{u_{j+10}v_{j+10}\}$. If $S_3\cap E(U')\subset \{u_{j-3}u_{j-1},v_{j-1}v_j\}$, then we can transform this case to the case of $|S_3\cap E(W')|=2$ by changing the edge in $S_3\cap E(U')$ into edge $u_{j+1}v_{j+1}$ and maintaining the other edges; if $S_3\cap E(U')=\{v_{j-3}v_{j-2}\}$, then $|S_3\cap E(U'')|\ge 1$ (for otherwise, $S_3$ is contained in another perfect matching in $\mathcal{M}_2(P(n))$). If $|S_3\cap E(U'')|\ge 2$, then we could obtain a forcing set of $M_4$ from $S_3$ by transforming all edges in $S_3\cap (E(U'')\cup E(U')\cup E(W'))$ to five edges $u_{j-8}u_{j-6}$, $u_{j-4}u_{j-2}$, $u_{j+4}u_{j+6}$, $v_{j+7}v_{j+8}$, $u_{j+9}u_{j+11}$ and maintaining the other edges; if $|S_3\cap E(U'')|= 1$, then $S_3\cap E(U'')\subset \{u_{j-7}u_{j-5},v_{j-7}v_{j-6},v_{j-5}v_{j-4}\}$. We could obtain a forcing set of $M_4$ from $S_3$ by transforming all edges in $S_3\cap (E(U'')\cup E(U')\cup E(W'))$ to four edges $u_{j-7}u_{j-5}$, $v_{j-7}v_{j-6}$, $u_{j+1}u_{j+3}$, $u_{j+9}u_{j+11}$ and maintaining the other edges. These imply $f(P(n),M_4)\le f(P(n),M_3)+1$. Similarly, if $|S_3\cap E(V')|=1$, then $f(P(n),M_4)\le f(P(n),M_3)+1$.

        \textbf{Case 3.4.} $k_3=0$. Then $S_3$ is contained in another perfect matching in $\mathcal{M}_2(P(n))$, a contradiction.
    \end{proof}

    Note that we can derive the following expression to obtain the above lemma.

    \[f(P(n),D^dC^c)=\lceil \frac { 6d+3c+\eta(d) }{ 12 } \rceil+\xi(d,c),\]
    where $\eta(d)=9$ if $d$ is even$,$ and $\eta(d)=12$ otherwise$;$ $\xi (d,c)=1$ if either $d=0$ and $c\equiv 1$ (mod 4), or $d=1$ and $c\equiv 2$ (mod 4), and $\xi (d,c)=0$ otherwise.

    \begin{proof}[Proof of Theorem \ref{thm2con}]
        The initial cases of $3\le n\le 36$ can be verified from Table \ref{tab}. From now on suppose $n\ge 37$. For two arbitrary perfect matchings $M_i$ and $M_j$ in $\mathcal{M}_2(P(n))$, we can first give a series of transformations as Lemma \ref{conclusion1} from $M_i$ (resp. $M_j$) to $D^{d_i}C^{c_i}$ (resp. $D^{d_j}C^{c_j}$) with the variation of forcing numbers during each transformation no more than one (where $d_i$ and $c_i$ (resp. $d_j$ and $c_j$) denote the number of $D$ and $C$ in $M_i$ (resp. $M_j$), respectively), then give a series of transformations from $D^{d_i}C^{c_i}$ to $D^{d_j}C^{c_j}$ as Lemma \ref{conclusion} with the variation of forcing numbers during each transformation no more than one. Then the theorem holds.
    \end{proof}

\end{document}